\documentclass{article}
\pdfoutput=1

\usepackage{amsmath, amsthm, amssymb}
\usepackage{hyperref}
\usepackage{libertine}
\usepackage{courier}
\usepackage{mathpazo}
\usepackage{mathtools}
\usepackage{natbib}
\usepackage{float}
\usepackage{authblk}
\usepackage{pifont}
\usepackage{tikz}
\usepackage{xcolor}

\hypersetup{
	colorlinks=true,
  linkcolor=red,          %
  citecolor=blue,         %
  filecolor=magenta,      %
  urlcolor=cyan           %
}

    \makeatletter
\def\@fnsymbol#1{\ensuremath{\ifcase#1\or \dagger\or \ddagger\or
   \mathsection\or \mathparagraph\or \|\or **\or \dagger\dagger
   \or \ddagger\ddagger \else\@ctrerr\fi}}
    \makeatother

\definecolor{darkgreen}{rgb}{0.2, 0.5, 0.2}

\DeclareMathOperator{\diag}{diag}
\DeclareMathOperator{\rank}{rank}
\DeclareMathOperator*{\argmin}{arg\,min}

\DeclareMathOperator{\ehull}{e-hull}
\DeclareMathOperator{\mhull}{m-hull}
\DeclareMathOperator{\spann}{span}

\newcommand{\GW}[1]{#1}
\newcommand{\GWS}[1]{#1}

\newcommand{\set}[1]{\left\{ #1 \right\}}
\newcommand{\eqdef}{\triangleq}
\newcommand{\R}{\mathbb{R}}
\newcommand{\N}{\mathbb{N}}
\newcommand{\unit}{\boldsymbol{1}}

\newcommand{\trn}{^\intercal}

\newcommand{\cmark}{{\color{darkgreen}\ding{51}}}
\newcommand{\xmark}{{\color{red}\ding{55}}}
\newcommand{\abs}[1]{\left| #1 \right|}
\newcommand{\nrm}[1]{\left\Vert #1 \right\Vert}
\newcommand{\PR}[2][]{\mathbb{P}_{#1}\left( #2 \right)}

\newcommand{\eps}{\varepsilon}
\newcommand{\stoch}{\mathfrak{s}}
\newcommand{\bigO}{\mathcal{O}}
\newcommand{\kl}[2]{D\left(#1 \middle| \middle| #2\right)}
\newcommand{\rev}{_{\mathsf{rev}}}
\newcommand{\sym}{_{\mathsf{sym}}}
\newcommand{\bis}{_{\mathsf{bis}}}
\newcommand{\iid}{_{\mathsf{iid}}}

\newtheorem{theorem}{Theorem}[section]
\newtheorem{lemma}{Lemma}[section]
\newtheorem{definition}{Definition}[section]
\newtheorem{proposition}{Proposition}[section]
\newtheorem{example}{Example}[section]
\newtheorem{corollary}{Corollary}[section]
\newtheorem{remark}{Remark}[section]

\providecommand{\keywords}[1]
{
  \small	
  \textbf{\textit{Keywords---}} #1
}

\title{\vspace{-2.0cm} Information Geometry \\ of Reversible Markov Chains}

\author[1]{Geoffrey Wolfer \thanks{email: geo-wolfer@m2.tuat.ac.jp. \\
International Research Fellow of Japan Society for the Promotion of Science. Supported in part by KAKENHI under Grant 21F20378.}}
\author[2]{Shun Watanabe \thanks{email: shunwata@cc.tuat.ac.jp. \\Supported in part by Japan Society for the Promotion of Science KAKENHI under Grant 20H02144.}}
\affil[1,2]{Department of Computer and Information Sciences \protect\\ Tokyo University of Agriculture and Technology}

\date{\today}

\begin{document}

\maketitle

\begin{abstract}
We analyze the information geometric structure of 
time reversibility for parametric families of irreducible transition kernels of Markov chains.
We define and characterize reversible exponential families of Markov kernels,
and show that irreducible and reversible 
Markov kernels form both a mixture family and, perhaps surprisingly, an exponential family in the set of all stochastic kernels.
We propose a parametrization of the entire manifold of reversible kernels,
and inspect reversible geodesics.
We define information projections onto the
reversible manifold, and derive closed-form expressions
for the e-projection and m-projection, 
along with Pythagorean identities with respect to information divergence,
leading to some new notion of reversiblization of Markov kernels.
We show the family of edge measures
pertaining to irreducible and reversible kernels also forms
an exponential family among distributions over pairs.
We further explore geometric properties of the reversible family,
by comparing them with other remarkable families of stochastic matrices.
Finally, we show that reversible kernels are, in a sense we define,
the minimal exponential family generated by the m-family of symmetric kernels,
and the smallest mixture family that comprises the e-family of memoryless kernels.

\end{abstract}


\clearpage
\tableofcontents

\section{Introduction}
\label{section:introduction}
Time reversibility is a fundamental property of many statistical laws of nature.
Inspired by \citet{schrodinger1931umkehrung},
Kolmogorov was the first \citep{dobrushin1988kolmogorov}, in his celebrated work \citep{kolmogorov1936theorie, kolmogorov1937umkehrbarkeit}, to investigate this notion in the context of Markov chains and diffusion processes.
Reversible chains also find 
numerous applications in computer science, 
for instance in queuing networks \citep{kelly2011reversibility}
or Markov Chain Monte Carlo sampling algorithms \citep{brooks2011handbook}. \GW{For instance, a random walk over a weighted network corresponds to a reversible Markov chains \citep[Section~3.2]{aldous2002reversible}}. 

Reversible Markov operators enjoy a considerably richer mathematical structure than their non-reversible counterparts, enabling a wide range of analytical tools and techniques.
Indeed, the significance of reversibility spans across 
surprisingly many areas of mathematics, from spectral theory \citep[Chapter~12]{levin2009markov} to abstract algebra \citep{pistone2013algebra}.
\GW{For instance, the mixing time of a reversible Markov chain, i.e. the time to guarantee closeness to stationarity, is controlled up to logarithmic factors by its absolute spectral gap (the difference of its two largest eigenvalues in magnitude).
The diversity of the existing tools and analyses prompts our first question of whether reversibility
can also be treated from an information geometry perspective}.

Through the lens of information geometry, the manifold of all irreducible Markov kernels forms both an exponential family (e-family) and a mixture family (m-family).
Our natural second question is whether we can find subfamilies of irreducible kernels that enjoy similar geometric properties,
or in other words, 
can we find submanifolds that are autoparallel with respect to affine connections of interest?
For instance, the set of doubly-stochastic matrices is known to form an m-family \citep{hayashi2014information},
while a tree model is an e-family of Markov kernels, 
if and only if it is an \emph{FSMX model} \citep{takeuchi2017information}.

In this article, we will answer these two questions, 
see that reversible irreducible Markov chains enjoy the structure of 
both exponential and mixture families, and explore their geometric properties.

\subsection{Related work}
\label{section:related-work}
The concept of exponential tilting of stochastic matrices using Perron-Frobenius (PF) theory can be traced back to the work of \citet{miller1961convexity}.
The large deviation theory for Markov chains, whose crown achievement is showing that the convex conjugate of the log-PF root of the tilted kernel essentially controls the large deviation rate was further developed by \citet{donsker1975asymptotic, gartner1977large, dembo2011large}. \citet{csiszar1987conditional} seem to be the first to recognize the exponential structure of the set of irreducible Markov kernels, in the context of information projections.
Independently,
\citet{ito1988} implicitly introduced the notion of \emph{asymptotic exponential families}, and exhibited irreducible Markov kernels as an example. 
\citet{takeuchi1998asymptotically} later formalized this definition (see also \citet{takeuchi2007exponential}), and \citet{takeuchi2017asymptotic} subsequently proved that exponential families and their asymptotic counterparts are equivalent.
\citet{nakagawa1993converse} formally defined the exponential family of irreducible Markov chains and \citet{nagaoka2017exponential} later gave a full treatment in the language of information geometry, proving its dually flat structure.
A notable collection of works has also explored the implications of this geometric structure for problems related to parameter estimation \citep{hayashi2014information}, hypothesis testing \citep{nakagawa1993converse, watanabe2017}, large deviation theory \citep{moulos2019optimal}, and hidden Markov models \citep{hayashi2019local,hayashi2021information}.

We refer the reader to \citet{levin2009markov} and \citet{amari2007methods} for thorough treatments of the theory of Markov chains and information geometry.

\subsection{Outline and main results}
\label{section:main-results}
In Section~\ref{section:preliminaries},
we begin with a primer on reversible Markov chains,
define exponential and mixture families,
and briefly discuss the importance of affine structures for our analysis of exponential families.
In Section~\ref{section:time-reversal-exponential-family},
we define a time-reversal operation on parametric families, 
and show in Proposition~\ref{proposition:time-reversal-e-family-is-e-family} that both m-families and e-families are closed under this transformation.
In Section~\ref{section:reversible-exponential-family}
we introduce the concept of a reversible e-family, and provide a characterization (Theorem~\ref{theorem:characterization-reversible-e-family}) of such family in terms of its 
carrier kernel and set of generator functions.
Adapting the Kolmogorov criterion, 
we show that the necessary and sufficient conditions 
can be verified in a time that depends polynomially on 
the number of states.
In Section~\ref{section:submanifold-reversible-kernels},
we prove that the set of all reversible and irreducible transition 
kernels is both an m-family, 
and an e-family (Theorem~\ref{theorem:quotient-space-dimension-e-family}), construct a basis (Theorem~\ref{theorem:reversible-basis}), and derive a parametrization (Theorem~\ref{theorem:explicit-parametrization}) of the entire set of reversible kernels.
In Section~\ref{section:information-projections},
we investigate information projections of 
an irreducible Markov chain onto its reversible submanifold.
We show that the projections verify Pythagorean identities, 
and obtain closed-form expressions (Theorem~\ref{theorem:e-m-projections}). Additionally,
we prove that the projections are always equidistant 
from an irreducible Markov kernel and its time-reversal (bisection property, Proposition~\ref{proposition:divergence-symmetry}).
In Section~\ref{section:edge-measure-family},
we show that reversible edge measures also form an e-family in distributions over pairs (Theorem~\ref{theorem:edge-measures-e-family}).
In Section~\ref{section:comparison-usual-classes},
we briefly compare the geometric properties of reversible chains with several other natural families of Markov kernels. 
Finally, in Section~\ref{section:generation-by-completion}, we
characterize the reversible family as both the smallest exponential family that comprises symmetric kernels (Theorem~\ref{theorem:e-hull-symmetric-is-reversible}), and the smallest mixture family that contains memoryless Markov kernels (Theorem~\ref{theorem:m-hull-iid-is-reversible}). 

\section{Preliminaries}
\label{section:preliminaries}
For $m \in \N$ we write $[m] = \set{1, 2, \dots, m}$.
Let $\mathcal{X}$ be a set such that $\abs{\mathcal{X}} = m < \infty$, identified with $[m]$, where to avoid trivialities, we also assume that $m > 1$.
We denote $\mathcal{P}(\mathcal{X})$ the probability simplex over $\mathcal{X}$, and $\mathcal{P}_{+}(\mathcal{X}) = \set{ \mu \in \mathcal{P}(\mathcal{X}) \colon \forall x \in \mathcal{X}, \mu(x) > 0}$. 
All vectors will be written as row-vectors, unless otherwise stated.
For some real matrices $A$ and $B$, $\rho(A)$ is 
the spectral radius of $A$, $f[A]$ for $f \colon \R \to \R$ is the entry-wise application of $f$ to $A$; $A \circ B$ is the Hadamard product of $A$ and $B$, $A > 0$ (resp. $A \geq 0$) means that $A$ is an entry-wise positive (resp. non-negative) matrix.
We will routinely identify a function $f \colon \mathcal{X}^2 \to \R$ with the linear operator $f \colon \R^ \mathcal{X} \to \R^\mathcal{X}$.

\subsection{Irreducible Markov chains}

\GW{We let $(\mathcal{X}, \mathcal{E})$ be a strongly connected directed graph, where $\mathcal{X}$ is the set of vertices, and $\mathcal{E} \text{ \GW{$\subset$} } \mathcal{X}^2$ the set of edges}. 
Let $\mathcal{F}(\mathcal{X}, \mathcal{E})$ be the set of all real functions over the set $\mathcal{E}$, identified with the totality of functions over $\mathcal{X}^2$ that are null outside of $\mathcal{E}$, and let $\mathcal{F}_+(\mathcal{X}, \mathcal{E}) \text{ \GW{$\subset$} } \mathcal{F}(\mathcal{X}, \mathcal{E})$ be the subset of positive functions over $\mathcal{E}$. 
Similarly, we define $\mathcal{P}(\mathcal{E}) = \mathcal{P}(\mathcal{X}^2) \cap \mathcal{F}_+(\mathcal{X}, \mathcal{E})$, the set of distributions whose mass is concentrated on the edge set $\mathcal{E}$.
We write $\mathcal{W}(\mathcal{X})$ for the set of row-stochastic transition kernels over the state space $\mathcal{X}$, and  $\mathcal{W}(\mathcal{X}, \mathcal{E})$ for the subset of irreducible kernels whose support is $\mathcal{E}$, i.e.
\begin{equation*}
\begin{split}
\mathcal{W}(\mathcal{X}) &\eqdef \set{ P \in \R^{\mathcal{X}^2} \colon P \geq 0, \forall x \in \mathcal{X}, \sum_{x' \in \mathcal{X}} P(x, x') = 1} ,\\
\mathcal{W}(\mathcal{X}, \mathcal{E}) &\eqdef \mathcal{F}_+(\mathcal{X}, \mathcal{E}) \cap \mathcal{W}(\mathcal{X}),
\end{split}
\end{equation*}
\GW{and where $P(x,x')$ corresponds to the transition probability from state $x$ to state $x'$ \footnote{\GW{We note that in information theory, $P(x,x')$ is often denoted by $P(x'|x)$. Our choice follows the applied probability literature (see e.g. \citet{levin2009markov}), and allows us to extend the notation seamlessly to general functions over $\mathcal{X}^2$.}}}.
For $P \in \mathcal{W}(\mathcal{X}, \mathcal{E})$, there exists a unique $\pi \in \mathcal{P}_+(\mathcal{X})$, such that $\pi P =\pi$ \citep[Corollary~1.17]{levin2009markov}, which we call the stationary distribution of $P$. 
When $\mathcal{E} = \mathcal{X}^2$ and if there is no ambiguity about the space under consideration, we may write more simply $\mathcal{F}, \mathcal{F}_+$ instead of $\mathcal{F}(\mathcal{X}, \mathcal{X}^2), \mathcal{F}_+(\mathcal{X}, \mathcal{X}^2)$ (a similar notation will apply to all subsequently defined spaces).

\subsection{Reversibility}

For an irreducible kernel $P$, we write 
$Q = \diag(\pi) P$ for the edge measure matrix, \citep[(7.5)]{levin2009markov}, which corresponds to stationary pair-probabilities of $P$, i.e. 
$Q(x, x') = \PR[\pi]{X_t = x, X_{t+1} = x'}$, and denote the set of irreducible edge measures by
$$\mathcal{Q}(\mathcal{X}, \mathcal{E}) \eqdef \set{ \diag(\pi) P \colon P \in \mathcal{W}(\mathcal{X}, \mathcal{E}), \pi P = \pi } \text{ \GW{$\subset$} } \mathcal{P}(\mathcal{E}).$$
Note that this definition is equivalent to
\begin{equation}
\label{eq:equivalent-edge-measure-set}
    \mathcal{Q}(\mathcal{X}, \mathcal{E}) = \set{ Q \in \mathcal{P}(\mathcal{E}) \colon  \sum_{x' \in \mathcal{X}}Q(x, x') = \sum_{x' \in \mathcal{X}}Q(x', x) } .
\end{equation}
We further denote $P^\star$ for the uniquely defined \emph{time-reversal} of $P$, that verifies $P^\star(x, x') = \pi(x') P(x', x) / \pi(x)$, \GW{and write $Q^\star = Q \trn$ for its corresponding edge measure, where $\trn$ denotes matrix transposition}.
When $Q$ is symmetric (i.e. $Q^\star = Q $), the chain verifies the \emph{detailed balance equation}, 
$$\pi(x) P(x, x') = \pi(x') P(x', x),$$
i.e. $P^\star = P$, 
and we say that the Markov chain is \emph{reversible}. 
Observe that in this case, for $P$ irreducible over $\mathcal{E}$, the edge set must also be symmetric ($\mathcal{E} = \mathcal{E}^\star$, where $\mathcal{E}^\star \eqdef \set{(x, x') \in \mathcal{X}^2 \colon (x' , x) \in \mathcal{E}}$).
We write $\mathcal{W} \rev (\mathcal{X}, \mathcal{E})$ for the set of all reversible kernels 
that are irreducible over $(\mathcal{X}, \mathcal{E})$.
For $f, g \in \R^\mathcal{X}$, $\langle f, g \rangle_\pi \eqdef \sum_{x \in \mathcal{X}} f(x) g(x) \pi(x)$ defines an inner product.
We call $\ell_2(\pi)$ the corresponding Hilbert space.
The time-reversal is the adjoint operator of $P$ in $\ell_2(\pi)$,
i.e. the unique linear operator that verifies
$\langle P f, g \rangle_\pi = \langle  f, P^\star g \rangle_\pi, \forall f,g \in \R^\mathcal{X}$ (represented here as column vectors).
As a consequence, when $P$ is reversible, it is also self-adjoint in $\ell_2(\pi)$, and the spectrum of $P$ is real.

\subsection{Mixture family and exponential family}

For later convenience we consider the following three equivalent definitions of a mixture family.

\begin{definition}[m-family of transition kernels]
\label{definition:m-family-parametric}

We say that a family of irreducible transition kernels $\mathcal{V}_m$ 
is a mixture family (m-family) of irreducible transition kernels on $(\mathcal{X}, \mathcal{E})$ when one of the following (equivalent) statements $(i), (ii), (iii)$ holds.
\begin{enumerate}
    \item[$(i)$]  \citep{fujiwara2015foundations}
    There exist affinely independent $Q_0, Q_1, \dots, Q_d \in \mathcal{Q}(\mathcal{X}, \mathcal{E})$ such that
    $$\mathcal{V}_m = \set{ P_\xi \in \mathcal{W}(\mathcal{X}, \mathcal{E}) \colon Q_\xi = \sum_{i = 1}^{d} \xi^i Q_i + (1 - \sum_{i=1}^{d}\xi^i) Q_0, \xi \in \Xi },$$
    where $\Xi = \set{ \xi \in \R^d \colon Q_\xi(x,x') > 0, \forall (x,x') \in \mathcal{E} }$, and $Q_\xi$ is the edge measure that pertains to $P_\xi$.
    \item[$(ii)$] \citep[2.35]{amari2007methods}
     There exists $C, F_1, \dots, F_d \in \mathcal{F}(\mathcal{X}, \mathcal{E})$,
such that $C, C + F_1, \dots, C + F_d$ are affinely independent, $$\sum_{x,x'} C(x,x') = 1, \qquad \sum_{x,x'} F_i(x,x') = 0, \forall i \in [d],$$
and 
$$\mathcal{V}_m = \set{ P_\xi \in \mathcal{W}(\mathcal{X}, \mathcal{E}) \colon Q_\xi = C + \sum_{i =1}^{d} \xi^i F_i, \xi \in \Xi }$$ 
where $\Xi= \set{ \xi \in \R^d\colon Q_\xi(x,x') > 0, \forall (x,x') \in \mathcal{E} }$, and $Q_\xi$ is the edge measure that pertains to $P_\xi$.
    \item[$(iii)$] \citep[Section~4.2]{hayashi2014information} There exist $k \in \N, g_1, \dots, g_k \in \mathcal{F}(\mathcal{X}, \mathcal{E})$ and $c_1, \dots, c_k \in \R$, such that
$$\mathcal{V}_m = \set{ P \in \mathcal{W}(\mathcal{X}, \mathcal{E}) \colon \sum_{x,x'} Q(x,x') g_i(x,x') = c_i, \forall i \in [k]}.$$
\end{enumerate}
Note that $\Xi$ is an open set, $\xi$ is called the mixture parameter and $d$ is the dimension of the family $\mathcal{V}_m$.
\end{definition}

\begin{definition}[e-family of transition kernels]
\label{definition:e-family-parametric}
Let $\Theta \text{ \GW{$\subset$} } \R^d$,  be some connected parameter space that contains an open ball centered at $0$.
We say that the parametric family of irreducible transition kernels $$\mathcal{V}_e = \set{P_\theta \colon \theta = (\theta^1, \dots, \theta^d) \in \Theta}$$ 
is 
an exponential family (e-family) of transition kernels on $(\mathcal{X}, \mathcal{E})$ with natural parameter $\theta$, whenever
\begin{enumerate}
    \item[$(i)$] For all $\theta \in \Theta$,  $P_\theta \in \mathcal{W}(\mathcal{X}, \mathcal{E})$ .
    \item[$(ii)$] There exist functions
    \begin{equation*}
        \begin{split}
            K \colon \mathcal{X} \times \mathcal{X} &\to \R, \\
            R \colon \Theta \times \mathcal{X} &\to \R, \\
            g_1, \dots, g_d \colon \mathcal{X} \times \mathcal{X} &\to \R, \\
            \psi \colon \Theta &\to \R,
        \end{split}
    \end{equation*}
    such that $\forall (x,x', \theta) \in \mathcal{X}^2 \times \Theta$,
    \GW{
    \begin{equation}
        \label{eq:e-family-expression}
        \log P_\theta(x, x') =  K(x, x') + \sum_{i = 1}^{d} \theta^i g_i(x, x') + R(\theta, x') - R(\theta, x)  - \psi(\theta),
    \end{equation}}
    when $(x, x') \in \mathcal{E}$, and $P_\theta(x, x') = 0$ otherwise.
\end{enumerate}
\end{definition}
When fixing some $\theta \in \Theta$, we may later write for convenience $\psi_\theta$ for $\psi(\theta)$ and  $R_\theta$ for $R(\theta, \cdot) \in \R^\mathcal{X}$.
The carrier kernel $K$, the collection of generator functions $g_1, \dots, g_d$ and the parameter range $\Theta$ define the family entirely. The remaining functions $R_\theta$ and $\psi_\theta$ will be determined uniquely by PF theory, from the constraint of $P_\theta$ being row-stochastic
(see for example the proof of Proposition~\ref{proposition:time-reversal-e-family-is-e-family}).
In fact, we can define the mapping $\stoch$
that constructs a proper irreducible stochastic matrix from any linear operator defined by an irreducible matrix over
$(\mathcal{X}, \mathcal{E})$.
\begin{equation}
\label{equation:stochastic-version}
\begin{split}
    \stoch \colon \mathcal{F}_+(\mathcal{X},\mathcal{E}) &\to \mathcal{W}(\mathcal{X},\mathcal{E}) \\
    \widetilde{P}(x,x') &\mapsto P(x,x') = \frac{\widetilde{P}(x,x') v(x')}{\rho(\widetilde{P})v(x)}, 
\end{split}
\end{equation}
where $\rho(\widetilde{P})$ 
and $v$ are respectively the
PF root and right PF eigenvector of $\widetilde{P}$. 

\begin{remark}
In \citet{feigin1981conditional, kuchler1998exponential, hudson1982large, stefanov1995explicit, kuchler1989exponential, sorensen1986sequential}, 
an exponential family of transition kernels has the form
\begin{equation*}
    \begin{split}
        \log P_\theta(x, x') =  K(x, x') + \sum_{i = 1}^{d} \theta^i g_i(x, x') - \phi(\theta, x'),
    \end{split}
\end{equation*}
for some function $\phi \colon \Theta \times \mathcal{X} \to \R$.
Our Definition~\ref{definition:e-family-parametric} however follows the one of \citet{nagaoka2017exponential, hayashi2014information, watanabe2017}, that is endowed with a more compelling geometrical structure \citep[Remark~3]{hayashi2014information}.
\end{remark}

Following the information geometry philosophy \citep{amari2007methods}, we view the e-families or m-families that we defined, as $d$-dimensional submanifolds of $\mathcal{W}(\mathcal{X}, \mathcal{E})$ with corresponding chart maps $\theta, \xi \colon \mathcal{W}(\mathcal{X}, \mathcal{E}) \to \R^d$.
We can give more geometrical, 
parametrization-free definitions of e-families and m-families of 
irreducible transition kernel over $(\mathcal{X}, \mathcal{E})$, 
as autoparallel submanifolds of $\mathcal{W}(\mathcal{X}, \mathcal{E})$ with respect to the e-connection and m-connection \citep[Section~6]{nagaoka2017exponential}.
We will prefer, however, to mostly cast our analysis in the language of linear algebra, and defer analysis of the relationship with differential geometry concepts to Section~\ref{section:doubly-autoparallel}.
This choice is motivated by the existence of a known correspondence between affine functions over $\mathcal{E}$ and the manifold 
$\mathcal{W}(\mathcal{X}, \mathcal{E})$ \citep{nagaoka2017exponential} that we now describe.
Denote, 
\begin{equation}
\label{definition:anti-shift-functions}
\begin{split}
\mathcal{N}(\mathcal{X}, \mathcal{E}) \eqdef \bigg\{ &h \in \mathcal{F}(\mathcal{X}, \mathcal{E}) \colon \exists (c, f) \in (\R, \R^\mathcal{X}), \\ &\forall (x, x') \in \mathcal{E}, h(x, x') = f(x') - f(x) + c \bigg\}.
\end{split}
\end{equation}
Then
$\mathcal{F}(\mathcal{X}, \mathcal{E})$ defines a $\abs{\mathcal{E}}$-dimension vector space, while
$\mathcal{N}(\mathcal{X}, \mathcal{E})$ is an $\abs{\mathcal{X}}$-dimensional vector space \citep[Section~3]{nagaoka2017exponential}.
Introducing the mapping,
\begin{equation}
\label{equation:delta-diffeo}
\begin{split}
\Delta \colon \mathcal{F}(\mathcal{X}, \mathcal{E}) &\to \mathcal{W}(\mathcal{X}, \mathcal{E}) \\  
f &\mapsto \text{\GW{$\Delta(f) =
\mathfrak{s}(\exp \circ f)$}},
\end{split}
\end{equation}
such that
We see from the expression at \eqref{equation:stochastic-version} that $\Delta$ gives a diffeomorphism
from the quotient linear space 
$$\mathcal{G}(\mathcal{X}, \mathcal{E}) \eqdef \mathcal{F}(\mathcal{X}, \mathcal{E})/\mathcal{N}(\mathcal{X}, \mathcal{E})$$ 
to 
$\mathcal{W}(\mathcal{X}, \mathcal{E})$ and
a subset $\mathcal{V}$ of $\mathcal{W}(\mathcal{X}, \mathcal{E})$ is an e-family if and only if there exists an affine subspace $\mathcal{A}$ of the quotient space $\mathcal{G}(\mathcal{X}, \mathcal{E})$
such that $\mathcal{V} = \Delta(\mathcal{A})$ 
(we identify a coset with a  representative function in that coset).
In this case, the correspondence is one-to-one, and the dimension of the affine space and the submanifold coincide \citep[Theorem~2]{nagaoka2017exponential}.
In particular, this entails that $\dim \mathcal{W}(\mathcal{X}, \mathcal{E}) = \abs{\mathcal{E}} - \abs{\mathcal{X}}$ \citet[Corollary~1]{nagaoka2017exponential}.

\begin{remark} 
For Definition~\ref{definition:e-family-parametric}, unless stated otherwise, we will henceforth assume that the $g_i$ form an independent family in $\mathcal{G}(\mathcal{X}, \mathcal{E})$. This will ensure that the family is well-behaved in the sense of \citet[Lemma~4.1]{hayashi2014information}.
\end{remark}

\section{Time-reversal of parametric families}
\label{section:time-reversal-exponential-family}
We begin by extending the definition of a time-reversal to families of Markov chains. 

\begin{definition}[Time-reversal family]
We say that the family of irreducible transition kernels $\mathcal{V}^\star$ is the time-reversal of the family of irreducible transition kernels $\mathcal{V}$ when $\mathcal{V}^\star = \set{ P^\star \colon P \in \mathcal{V}}$,
where $P^\star$ denotes the time-reversal of $P$.
\end{definition}

We now state the fundamental fact that the quality of being an e-family or an m-family of transition kernels is 
closed under this time-reversal operation.

\begin{proposition}
\label{proposition:time-reversal-e-family-is-e-family}
The following statements hold.
\paragraph{Time reversal of m-family:}
Let $\mathcal{V}_m$ be an m-family over $(\mathcal{X}, \mathcal{E})$, then $\mathcal{V}_m^\star$ is an m-family over $(\mathcal{X}, \mathcal{E}^\star)$.
Furthermore, if $\mathcal{V}_m$ is the m-family generated by $Q_1, \dots, Q_d \in \mathcal{Q}(\mathcal{X}, \mathcal{E})$ (following the notation at Definition~\ref{definition:m-family-parametric}-$(i)$),
then the time-reversal m-family is given by 

$$\mathcal{V}_m^\star = \set{ P_\xi \in \mathcal{W}(\mathcal{X}, \mathcal{E}^\star) \colon Q_\xi = \sum_{i = 1}^{d} \xi^i Q_i^\star + (1 - \sum_{i=1}^{d}\xi_i) Q_0^\star, \xi \in \Xi^\star },$$
where $Q_\xi$ pertains to $P_\xi$ and
    with
    $$\Xi^\star = \set{\xi \in \R^d \colon Q_\xi(x,x') > 0, \forall (x,x') \in \mathcal{E}^\star} = \Xi.$$

\paragraph{Time reversal of e-family:}
Let $\mathcal{V}_e$ be an e-family over $(\mathcal{X}, \mathcal{E})$, then $\mathcal{V}_e^\star$ is an e-family over $(\mathcal{X}, \mathcal{E}^\star)$.
Furthermore, if $\mathcal{V}_e$ is the e-family generated by $K$ and $g_1, \dots, g_d$ (following the notation at Definition~\ref{definition:e-family-parametric}),
then the time-reversal e-family is given by $\mathcal{V}^\star = \set{P^\star_\theta : \theta \in \Theta}$ such that 
\begin{equation*}
\log P^\star_\theta(x, x') =  K(x', x) + \sum_{i = 1}^{d} \theta^i g_i(x', x) + L_\theta(x') - L_\theta(x)  - \psi_\theta,
    \end{equation*}
when $(x, x') \in \mathcal{E}^\star$, $P_\theta^\star(x, x') = 0$ otherwise, and where $L_\theta$ is the left PF eigenvector of the non-negative irreducible matrix 
$$\widetilde{P}_\theta^\star(x,x') = \exp \left( K(x', x) + \sum_{i = 1}^{d} \theta^i g_i(x', x) \right).$$ 
\end{proposition}

\begin{proof}
Since the edge measure $Q^\star_\xi$ of the time-reversal $P^\star_\xi$ is the transpose of $Q_\xi$ corresponding to $P_\xi$, 
it is easy to obtain the expression of the time-reversal, and to see that $\mathcal{V}_m^\star$ is a mixture family.
It remains to show that this also holds true for e-families.
From the definition of an exponential family \GW{\eqref{eq:e-family-expression}},
and the requirement that $P_\theta$ be row-stochastic, 
it must be that for any $x \in \mathcal{X}$,
\begin{equation*}
\sum_{x' \in \mathcal{X}} \exp(K(x, x') + \sum_{i = 1}^{d}\theta^i g_i(x, x')) e^{R_\theta(x')} = e^{\psi_\theta} e^{R_\theta(x)},
\end{equation*}
or more concisely, writing $\widetilde{P}_\theta(x,x') = \exp(K(x, x') + \sum_{i = 1}^{d}\theta^i g_i(x, x'))$ for $x, x' \in \mathcal{E}$ and $\widetilde{P}_\theta(x,x') = 0$ otherwise, 
$\widetilde{P}_\theta \exp[R_\theta] = e^{\psi_\theta} \exp [R_\theta]$.
By positivity of the exponential function, the vector $\exp [R_\theta] \in \R^{\mathcal{X}}$ is positive. 
Thus, from  the PF theorem, $e^{\psi_\theta}$ corresponds to the spectral radius of $\widetilde{P}_\theta$, and $\exp[R_\theta]$ its (right) associated eigenvector.
There must therefore also exist a left positive eigenvector, which we denote by $\exp[L_\theta]$,
such that 
$$\exp[L_\theta] \widetilde{P}_\theta = e^{\psi_\theta} \exp[L_\theta].$$
Defining the positive normalized measure
\begin{equation}
\label{equation:parametric-stationary}
    \begin{split}
    \pi_\theta(x) \eqdef \frac{\exp(L_\theta(x) + R_\theta(x))}{\sum_{x'' \in \mathcal{X}} \exp(L_\theta(x'') + R_\theta(x''))},
    \end{split}
\end{equation}
it is easily verified that $\pi_\theta$ is the stationary distribution of $P_\theta$.
Notice that $\theta$, $K$ and $g_i$ determine uniquely 
$L_\theta, R_\theta, \psi_\theta $ and $\pi_\theta$ by the PF theorem.
Recall that the adjoint of a transition kernel $P$ can be written $P^\star(x, x') = \pi(x')P(x', x)/\pi(x)$, thus we can compute the time-reversal as
\begin{equation*}
\begin{split}
P_\theta(x', x) \frac{\pi_\theta(x')}{\pi_\theta(x)} 
&= \exp\left( K(x', x) + L_\theta(x') - L_\theta(x) + \sum_{i = 1}^{d}\theta^i g_i(x', x) - \psi_\theta \right), \\
\end{split}
\end{equation*}
when $(x', x) \in \mathcal{E}$, and $0$ for $(x', x) \not \in \mathcal{E}$.
The requirements of Definition~\ref{definition:e-family-parametric} for an e-family are all fulfilled,
which concludes the proof.
\end{proof}

\begin{remark}
\GW{Recall that for a distribution $\mu \in \mathcal{P}(\mathcal{X})$, we can by \emph{exponential change of measure} -- also known as \emph{exponential tilting} -- construct the natural exponential family of $\mu$:
$$\mu_\theta(x) = \mu(x)\exp( \theta x- A(\theta)),$$
where $A(\theta)$ is a normalization function that ensures $\mu_\theta \in \mathcal{P}(\mathcal{X})$ for all $\theta \in \R$. 
The idea of exponential change of measure for distributions can be traced back to Chernoff \citep{chernoff1952measure}, and was later termed tilting \citep{gallager1968information, van1981maximum}.
Similarly, given some function $g \colon \mathcal{X}^2 \to \R$ we can tilt an irreducible kernel $P$ (e.g. \citet{miller1961convexity}), by first constructing $\tilde{P}_\theta(x,x') = P(x,x') e^{\theta g(x,x')}$, and then rescaling the newly obtained irreducible matrix \footnote{\GW{Interestingly, 
 the large deviation rate of $g$ is given by the convex conjugate 
of the log-PF root of $\tilde{P}_\theta$ \citep[Chapter~3]{dembo2011large}.}} with the mapping $\stoch$. When $\theta = 0$, notice that we recover the original $P$.
}
But while in our definition,
$$P_\theta = \stoch (P \circ \exp[\theta g]) =  \frac{1}{\rho(P \circ \exp[\theta g])} \diag v_\theta^{-1} (P \circ \exp[\theta g]) \diag v_\theta$$ 
denotes the kernel tilted \GW{involving}
the right PF eigenvector $v_\theta$, we could alternatively define the Markov kernel $P'_\theta$ by tilting $P$ \GW{with} the left PF eigenvector $u_\theta$:
$$P'_\theta = \frac{1}{\rho(P \circ \exp[\theta g])} \diag u_\theta^{-1} (P \circ \exp[\theta g]) \trn \diag u_\theta.$$ 
Observe that the right and left tilted versions of $P$ with
identical $\theta$ share the same stationary distribution $\pi_\theta \propto u_\theta \circ v_\theta$  \eqref{equation:parametric-stationary}
and that they are in fact each other's time-reversal ($P'_\theta = P^\star_\theta$), i.e. they form a pair of adjoint linear operators over the space $\ell_2(\pi_\theta)$.
\end{remark}

\section{Reversible exponential families}
\label{section:reversible-exponential-family}
The previous section extended the time-reversal operation to parametric 
families of transition kernels. It seems then natural to investigate fixed points, i.e. 
parametric families that remain invariant under this transformation.
We say that an irreducible e-family $\mathcal{V}(\mathcal{X}, \mathcal{E})$ is reversible when $P$ is reversible $\forall P \in \mathcal{V}(\mathcal{X}, \mathcal{E})$.
In this case, $\mathcal{E}$ coincides with $\mathcal{E}^\star$  and $\mathcal{V}^\star(\mathcal{X}, \mathcal{E}^\star)$ with $\mathcal{V}(\mathcal{X}, \mathcal{E})$.
Observe first that an e-family obtained from tilting a reversible $P$
is not generally reversible, making it clear that the reversible nature 
of the family cannot be determined solely by the properties of the carrier kernel $K$.
It is however easy to see that an e-family is reversible when $K$ and all the generator functions $g_1, \dots, g_d$ are symmetric. 
Moreover, for a state space $\mathcal{X}$ of size $2$,
    any exponential family would be reversible regardless of symmetry, 
    showing that this condition is not always necessary.
In this section, we give a complete characterization of this invariant set. 
Additionally, we explore the algorithmic cost of checking whether this property is verified from the description of the carrier kernel and generators of a given e-family.
Before diving into the general theory of reversible e-families, let us consider the following simple examples.
\begin{example}[Lazy random walk on the $m$-cycle] \label{example:lazy-random-walk}
\GW{For ${\cal X}=[m]$, and 
$$\mathcal{E}_{\mathsf{cy}} = \set{ (i,j) \in [m]^2 \colon \abs{i - j} \mod m - 1 \in \set{0, 1} },$$}
let 
\begin{equation*}
\begin{split}
P_\theta(x,x^\prime) &= \exp\bigg( \theta_1 \sum_{i=1}^m \delta_i(x)\delta_i(x^\prime) 
\\ &+ \theta_2 \sum_{i=1}^m \big(\delta_i(x)\delta_{i+1}(x^\prime) - \delta_{i+1}(x)\delta_i(x^\prime) \big) - \psi(\theta) \bigg),    
\end{split}
\end{equation*}
where $\psi(\theta)=\log\big( e^{\theta_1}+e^{\theta_2} + e^{-\theta_2}\big)$ and $\delta_{m+1}=\delta_1$.
This e-family $\{P_\theta \}_{\theta \in \mathbb{R}^2}$ corresponds to the set of biased lazy random walks on the $m$-cycle given by
\begin{align*}
P_\theta(x,x) &= \frac{e^{\theta_1}}{e^{\theta_1}+e^{\theta_2}+e^{-\theta_2}}, \\
P_\theta(x,x+1) &= \frac{e^{\theta_2}}{e^{\theta_1}+e^{\theta_2}+e^{-\theta_2}}, \\
P_\theta(x+1,x) &= \frac{e^{-\theta_2}}{e^{\theta_1}+e^{\theta_2}+e^{-\theta_2}}.
\end{align*}
Observe that \GW{$P^\star_{(\theta_1, \theta_2)} = P_{(\theta_1, - \theta_2)}$, and thus} $P_\theta$ is not a reversible e-family. The subfamily $\{ P_\theta : \theta_1 \in \mathbb{R}, \theta_2=0 \}$, however,
i.e. unbiased lazy random walks on the $m$-cycle, \GW{form a}
a reversible e-family.
\end{example}

\begin{example}[Birth-and-death chains]
For ${\cal X}=[m]$ and ${\cal E}_{\mathsf{bd}}=\{ (i,j) : |i-j| \le 1\}$, a Markov kernel having its support
on ${\cal E}_{\mathsf{bd}}$ is referred to as a birth-and-death chain. Since every birth-and-death chain is
reversible \cite[Section 2.5]{levin2009markov}, ${\cal W}({\cal X},{\cal E}_{\mathsf{bd}})$ is a reversible e-family. 
\end{example}

We first recall Kolmogorov's characterization of reversibility, 
which will be instrumental in our argument.
For $\mathcal{E} \text{ \GW{$\subset$} } \mathcal{X}^2$ such that $(\mathcal{X}, \mathcal{E})$ is a strongly connected directed graph, we write $\Gamma(\mathcal{X}, \mathcal{E})$ for the set of \emph{finite directed closed paths} in the graph $(\mathcal{X}, \mathcal{E})$. 
\GW{Formally, we treat $\gamma$ as a map $[n] \to \mathcal{E}$ such that $\gamma(t) = (x_t, x_{t+1})$ with $x_{n+1} = x_1$ and we write $\abs{\gamma} = n$ for the length of the path. For each $\gamma \in \Gamma(\mathcal{X}, \mathcal{E})$, we also introduce the reverse closed path $\gamma^\star \in \Gamma(\mathcal{X}, \mathcal{E}^\star)$ given by $\gamma^\star(t) = (x_{t+1}, x_t)$.
Namely, if $\gamma \in \Gamma(\mathcal{X}, \mathcal{E})$,
we can write $\gamma$ informally as a succession of edges such that the starting and finishing states agree (i.e. as an element of $\mathcal{E}^n$). 
$$\gamma = ((x_1, x_2), (x_2, x_3), \dots, (x_{n - 1}, x_n) ,(x_n, x_1) ), x_i \in \mathcal{X}, \forall i \in [n].$$
Note that $\gamma$ is not necessarily a \emph{cycle}, i.e. in our definition, multiple occurrences of the same point of the space are allowed.}

\begin{theorem}[Kolmogorov's criterion \citep{kolmogorov1936theorie}]
\label{theorem:kolmgorov-criterion}
Let $P$ irreducible over $(\mathcal{X}, \mathcal{E})$.
$P$ is reversible if and only if for all $\gamma \in \Gamma(\mathcal{X}, \mathcal{E})$,
\GW{$$\prod_{t = 1 }^{\abs{\gamma}} P(\gamma(t)) = \prod_{t = 1}^{\abs{\gamma}} P(\gamma^\star(t)).$$}
\end{theorem}

\begin{example}
When $\abs{\mathcal{X}} = 2$, all chains are reversible. 
For  $\abs{\mathcal{X}} = 3$, only one 
equation needs to be verified for $P$ to be reversible: 
$$P(1,2)P(2,3)P(3,1) = P(1,3)P(3,2)P(2,1).$$
\end{example}

We now extend the definition of reversibility to arbitrary irreducible functions $\mathcal{F}_+(\mathcal{X}, \mathcal{E})$ (non-negative on $\mathcal{X}^2$ and positive exactly on $\mathcal{E}$) based on Kolmogorov's criterion, and further introduce the concept of log-reversibility for $\mathcal{F}(\mathcal{X}, \mathcal{E})$, 
that considers sums instead of products.

\begin{definition}[Reversible and log-reversible functions]
\label{definition:revesible-function}
Let $\mathcal{E} \text{ \GW{$\subset$} } \mathcal{X}^2$ such that $\mathcal{E} = \mathcal{E}^\star$. 
\begin{description}
    \item[reversible:] A function $h \in \mathcal{F}_+(\mathcal{X}, \mathcal{E})$ is reversible whenever it satisfies that,
\GW{$$\prod_{t = 1 }^{\abs{\gamma}} h(\gamma(t)) = \prod_{t = 1}^{\abs{\gamma}} h(\gamma^\star(t)).$$}
for all finite directed closed paths $\gamma \in \Gamma(\mathcal{X}, \mathcal{E})$.
    \item[log-reversible:] A function $h \in \mathcal{F}(\mathcal{X}, \mathcal{E})$ is log-reversible whenever it satisfies that,
\GW{$$\sum_{t = 1 }^{\abs{\gamma}} h(\gamma(t)) = \sum_{t = 1}^{\abs{\gamma}} h(\gamma^\star(t)).$$}
for all finite directed closed paths $\gamma \in \Gamma(\mathcal{X}, \mathcal{E})$.
\end{description}

\end{definition}
Remark: These definitions do not rely on 
connectedness properties of $\mathcal{E}$ per se, 
but we will assume irreducibility nonetheless.
Observe that when $h$ is represented by an irreducible row-stochastic matrix,
the definition of reversibility of $h$ as a function
and as a Markov operator coincide
by Kolmogorov's criterion (Theorem~\ref{theorem:kolmgorov-criterion}).
Clearly, for $h \in \mathcal{F}(\mathcal{X}, \mathcal{E})$, $\exp[h]$ being reversible is equivalent to $h$ being log-reversible.
We could endow the set of positive 
reversible functions on $\mathcal{E}$ with 
a group structure by
considering the standard multiplicative operation on functions.
We will choose however (Lemma~\ref{lemma:subspace-log-reversible}), 
to rather construct and focus on the vector space of log-reversible functions.

\begin{lemma}
\label{lemma:rank-1-reversible}
Let \GW{$h \in \mathcal{F}_+(\cal X, \mathcal{E})$} 
such that $\rank(h) = 1$. Then $h$ is a reversible function.
\end{lemma}

\begin{proof}
\GW{
Consider a closed path $\gamma \in \Gamma(\mathcal{X}, \mathcal{E})$. Writing $h = u_h \trn v_h$, 
we successively have that
\begin{equation*}
\begin{split}
\prod_{t = 1 }^{\abs{\gamma}} h(\gamma(t)) &= \prod_{t = 1}^{\abs{\gamma}} u_h(x_t)v_h(x_{t+1}) = \prod_{t = 1}^{\abs{\gamma}} u_h(x_t)  \prod_{t = 1}^{\abs{\gamma}} v_h(x_t) = \prod_{t = 1}^{\abs{\gamma}} h(\gamma^\star(t)).\\
\end{split}
\end{equation*}
}
\end{proof}

\begin{theorem}[Characterization of reversible e-family]
\label{theorem:characterization-reversible-e-family}
Let $\mathcal{V}(\mathcal{X}, \mathcal{E})$ be an irreducible e-family of Markov chains, with natural parametrization $\theta$, generated by $K$ and $(g_i)_{i \in [d]}$. 
The following two statements are equivalent.
\begin{enumerate}
    \item[$(i)$] $\mathcal{V}(\mathcal{X}, \mathcal{E})$ is reversible.
    \item[$(ii)$] $\mathcal{E} = \mathcal{E}^\star$ and $\mathcal{V}(\mathcal{X}, \mathcal{E})$ is such 
    that the carrier kernel $K$ and generator functions $g_i, \forall i \in [d]$ are all log-reversible functions.
\end{enumerate}
\end{theorem}

\begin{proof}
We apply Kolmogorov's criterion to some arbitrary family member.
Let $\gamma$ be some finite closed path in $(\mathcal{X}, \mathcal{E})$,
\GW{$$\prod_{t = 1 }^{\abs{\gamma}} P_\theta(\gamma(t)) = \prod_{t = 1}^{\abs{\gamma}} P_\theta(\gamma^\star(t)).$$}
Rewriting the left-hand side,
\GW{
\begin{equation*}
\begin{split}
&\prod_{t =1 }^{\abs{\gamma}} P_\theta(\gamma(t)) \\
&= \prod_{t = 1}^{\abs{\gamma}} \exp \left( K(x_t, x_{t+1}) + R(\theta, x_{t+1}) - R(\theta, x_t) + \sum_{i = 1}^{d}\theta^i g_i(x_t, x_{t+1}) - \psi(\theta) \right)  \\
&= \exp(-\abs{\gamma}\psi(\theta)) \exp \left( \sum_{t = 1}^{\abs{\gamma}} \left[ K(x_t, x_{t+1}) + \sum_{i = 1}^{d}\theta^i g_i(x_t, x_{t+1}) \right] \right).   \\
\end{split}
\end{equation*}
}
Proceeding in a similar way with the right-hand side, we obtain
\GW{\begin{equation*}
\begin{split}
 &\sum_{t =1 }^{\abs{\gamma}} K(x_t, x_{t+1}) + \sum_{i = 1}^{d}\theta^i  \sum_{t =1 }^{\abs{\gamma}}  g_i(x_t, x_{t+1}) \\
 =  &\sum_{t = 1}^{\abs{\gamma}} K(x_{t+1}, x_{t}) + \sum_{i = 1}^{d}\theta^i  \sum_{t = 1}^{\abs{\gamma}}  g_i(x_{t+1}, x_{t}). \\
\end{split}
\end{equation*}}
When $K$ and the $g_i$ are log-reversible, this equality is verified for any closed path, and every member of the family is therefore reversible.
Taking $\theta = 0$ yields the reversibility requirement for $K$.
Further taking $\theta^i = \delta_{i}(j)$ for $j \in [d]$ similarly yields the requirement
for $g_j$. 
\end{proof}

This path checking approach, although mathematically convenient, 
is not algorithmically efficient.
In order to determine whether a full-support kernel --or function--
is reversible, the number of distinct Kolmogorov equations that must be checked is
\begin{equation*}
        \sum_{k = 3}^{\abs{\mathcal{X}}} {\binom{\abs{\mathcal{X}}} {k} } \frac{(k - 1)!}{2}, \qquad \text{\citep[Proposition~2.1]{jiang2018reversibility}}
\end{equation*}
which corresponds to the maximal number of cycles \GW{(i.e. closed paths such that the only repeated vertices are the first and last one)} in a complete graph over $\abs{\mathcal{X}}$ nodes. %
Such testing algorithm becomes rapidly 
intractable as $\abs{\mathcal{X}}$ increases.
However for Markov kernels, we know that this 
is equivalent 
to verifying the detailed balance equation, 
which can be achieved in (at most) polynomial time  $\bigO(\abs{\mathcal{X}}^3)$, by solving a linear system in order to find $\pi$.
We show that this idea naturally extends to verifying 
reversibility of functions, enabling us to design an algorithm of 
time complexity $\bigO(\abs{\mathcal{X}}^3)$.

\begin{lemma}
\label{lemma:equivalence-reversibility}
Let $h \in \mathcal{F}_+(\mathcal{X}, \mathcal{E})$ irreducible. $h$ is reversible if and only if $\Pi_h \circ h \trn$ is a symmetric matrix, with
$\Pi_h = v_h \trn u_h$ the PF projection of $h$,
where $u_h $ and $v_h$ are respectively the left and right PF eigenvectors of $h$, normalized such that $u_h v_h \trn = 1$.
\end{lemma}

\begin{proof}
Treat $h$ as the linear operator $h \colon \R^{\mathcal{X}} \to \R^{\mathcal{X}}$.
Suppose first that $h$ is reversible.
We apply PF theory, which guarantees 
that the following Ces\`{a}ro averages converge \citep[Example~8.3.2]{meyer2000matrix}
to some positive projection,
\begin{equation}
\label{equation:cesaro-averages}
    \lim_{n \to \infty} \frac{1}{n} \sum_{k = 0}^{n} \left(\frac{h}{\rho(h)} \right)^k = \Pi_h = v_h \trn u_h,
\end{equation}
Fix $(x, x') \in \mathcal{X}^2$ such that $h(x, x') \neq 0$. 
For $k \in \N$, we write $\Gamma_k(x,x') \text{ \GW{$\subset$} } \Gamma(\mathcal{X}, \mathcal{E})$ the set of all directed closed paths $\gamma$ with $(x_1, x_2, \dots, x_k) \in \mathcal{X}^k$ such that
$$\gamma = ((x, x_1), (x_1, x_2), \dots, (x_{k-1}, x_k), (x_k, x'), (x', x)).$$
For any such cycle, it holds (perhaps vacuously if $\Gamma_k(x,x') = \emptyset$) that
\begin{equation*}
\begin{split}
h(x, x_1) \cdots h(x_{k}, x') h(x', x) &= h(x, x')h(x', x_k)  \cdots  h(x_1, x).
\end{split}
\end{equation*}
Summing this equality over all possible paths in $\Gamma_k(x, x')$ \GW{(i.e. summing over all $(x_1, \dots, x_k) \in \mathcal{X}^k$, with the assumption that $h(x,x') = 0$ whenever $(x,x') \not \in \mathcal{E}$)}, we obtain
\begin{equation*}
    h^k(x, x') h(x', x) = h(x, x')h^k(x', x).
\end{equation*}
In the case where $h(x, x') = 0$, the above equation holds by symmetry of $\mathcal{E}$.
For $n \in \N$, appropriately rescaling on both sides with the PF root, summing over all $k \in \set{0, \dots ,n}$
and taking the limit at $n \to \infty$, \eqref{equation:cesaro-averages} yields 
detailed balance equations with respect to the projection $\Pi_h$,
\begin{equation*}
\begin{split}
    \Pi_h(x, x') h(x', x) &= h(x, x') \Pi_h(x', x), \\
\end{split}
\end{equation*}
or in other words, reversibility of $h$ implies symmetry of $\Pi_h \circ h \trn$.

To prove necessity, we suppose now that this symmetry holds, with $\Pi_h$
the PF projection of $h$.
We know that $\rank(\Pi_h) = \rank(v_h \trn u_h) = 1$, and that $\Pi_h$ is positive.
Consider some finite directed closed path $\gamma$. Rearranging products yields
\GW{
\begin{equation*}
\begin{split}
    \prod_{t = 1 }^{\abs{\gamma}} h(\gamma(t)) &= \left( \prod_{t = 1 }^{\abs{\gamma}} \frac{\Pi_h(\gamma(t))}{\Pi_h(\gamma^\star(t))} \right) \prod_{t = 1 }^{\abs{\gamma}} h(\gamma^\star(t)),  \\
\end{split}
\end{equation*}
}
but the first factor on the right-hand side vanishes,
from the fact that rank one functions are always reversible (Lemma~\ref{lemma:rank-1-reversible}). This concludes the proof of the lemma.
\end{proof}

Notice that we can define $\pi_h(x) \triangleq u_h(x)/v_h(x)$ the positive entry-wise ratio of the PF eigenvectors. 
We can then restate
Lemma~\ref{lemma:equivalence-reversibility}
in terms of the familiar detailed balance 
equation $\pi_h(x) h(x, x') = \pi_h(x') h(x', x)$.
\begin{corollary}
\label{corollary:easy-log-reversibility}
Let $h \in \mathcal{F}(\mathcal{X}, \mathcal{E})$ some irreducible function. 
$h$ is log-reversible if and only if there exists $f \in \R^\mathcal{X}$ 
such that $\forall x, x' \in \mathcal{X}$, $h(x, x') = h(x', x) + f(x') - f(x)$.
\end{corollary}
Remark: when $h$ is known to be reversible, one can compute $\pi_h$ in $\bigO(\abs{\mathcal{X}})$, by adapting the technique of \citep{suomela1979invariant}; 
unfortunately, it is not possible to check for 
reversibility using this method.
If the space becomes large, the reader can 
consider iterative (power) methods to compute 
the PF projector, potentially further reducing the 
verification time cost.
We end this section with a technical lemma that will allow us
in later sections to swiftly compute expectations of functions 
under certain reversibility or skew-symmetricity properties. 

\begin{lemma}
\label{lemma:q-expectation}
Let $P$ irreducible with associated edge measure matrix $Q$.
For a function $g \colon \mathcal{X}^2 \to \R$, we write
$Q[g] = \sum_{x, x' \in \mathcal{X}} Q(x, x') g(x, x')$.
\begin{enumerate}
    \item[$(i)$] If $g$ is log-reversible, $Q[g] = Q^\star[g]$. 
    \item[$(ii)$] If $g$ is skew-symmetric and $P$ is reversible, $Q[g] = 0$.
    \item[$(iii)$] If there exists $f \in \R^\mathcal{X}$ such that for all $x, x' \in \mathcal{X}$, $g(x,x') = f(x') - f(x)$, $Q[g] = 0$ (regardless of $P$ being reversible).
\end{enumerate}
\end{lemma}
\begin{proof}
\GW{Claim $(iii)$ follows by property of edge measure $Q$.}
\begin{equation*}
\begin{split}
    Q[g] &= \sum_{x, x' \in \mathcal{X}} Q(x, x')(f(x') - f(x))  \\
    &= \sum_{x' \in \mathcal{X}} f(x') \sum_{x \in \mathcal{X}} Q(x, x')  - \sum_{x \in \mathcal{X}} f(x) \sum_{x' \in \mathcal{X}} Q(x, x')  \\
    &= \sum_{x' \in \mathcal{X}} f(x') \pi(x')  - \sum_{x \in \mathcal{X}} f(x) \pi( x)  = 0.\\
\end{split}
\end{equation*}
\GW{From Corollary~\ref{corollary:easy-log-reversibility}, claim $(iii)$, and re-indexing,
\begin{equation*}
    \begin{split}
        Q[g] = Q[g \trn] + \sum_{x,x' \in \mathcal{X}} Q(x,x')(f(x') - f(x)) = Q^\star[g],
    \end{split}
\end{equation*}
which yields $(i)$}.
To prove $(ii)$, consider $g$ such that $g(x', x) = - g(x, x')$. Then
by re-indexing and symmetry of $Q$,
\begin{equation*}
\begin{split}
    Q[g] &= \sum_{x', x \in \mathcal{X}} Q(x', x) g(x', x)  = - \sum_{x', x \in \mathcal{X}} Q(x, x') g(x, x') = - Q[g].\\
\end{split}
\end{equation*}

\end{proof}

\section{The e-family of reversible Markov kernels}
\label{section:submanifold-reversible-kernels}

In Section~\ref{section:affine-structures}, we begin by analyzing the affine structure of the space of
log-reversible functions, derive its dimension, construct a basis, and deduce that the manifold of all irreducible reversible Markov kernels forms an exponential family.
The dimension of this family confirms the well-known fact 
that the number of free parameters for a 
reversible kernel is only about half 
of what is required for the general case, 
hence that reversible chains serve in a sense as a ``natural intermediate'' \citep[Section~5]{diaconis2006bayesian} in terms of model complexity. In Section~\ref{section:parametrization-reversible-kernels}, 
we proceed to derive a systematic parametrization of the manifold $\mathcal{W}(\mathcal{X}, \mathcal{E})$, similar in spirit to the one given in \citet{ito1988}, and in \citet[Example~1]{nagaoka2017exponential}.
In Section~\ref{section:doubly-autoparallel},
we connect our results to general differential geometry, and point out that reversible kernels $\mathcal{W} \rev(\mathcal{X}, \mathcal{E})$ form a doubly autoparallel submanifold in $\mathcal{W} (\mathcal{X}, \mathcal{E})$. Finally, we conclude with a brief discussion on reversible geodesics (Section~\ref{section:reversible-geodesics}).

\subsection{Affine structures}
\label{section:affine-structures}

Identifying $\mathcal{X}$ with $[m]$, we can endow the set with the natural order induced from $\N$. 
In this section, we will henceforth assume that $\mathcal{E}$ is symmetric, and consider the following subsets of $\mathcal{E}$,
\begin{equation*}
\begin{split}
T_-(\mathcal{E}) &\eqdef \set{(x, x') \in \mathcal{E} \colon x' > x }, \\
T_+(\mathcal{E}) &\eqdef \set{(x, x') \in \mathcal{E} \colon x' < x }, \\
T_0(\mathcal{E}) &\eqdef \set{(x, x') \in \mathcal{E} \colon x' =  x }, \\
\end{split}
\end{equation*}
and
\begin{equation*}
\begin{split}
T(\mathcal{E}) &\eqdef \set{(x, x') \in \mathcal{E} \colon x' \leq x, (x, x') \neq (m, x_\star), x_\star = \argmin_{x \in \mathcal{X}} \set{ (m, x) \in \mathcal{E} }}.
\end{split}
\end{equation*}
We immediately observe that the following cardinality relations hold
\begin{equation}
\label{equation:T-cardinalities}
\begin{split}
\abs{T_+(\mathcal{E})} &= \abs{T_-(\mathcal{E})}, \\
\abs{T(\mathcal{E})} &= \abs{T_+(\mathcal{E})} + \abs{T_0(\mathcal{E})} - 1 \text{ \GW{$ = \frac{\abs{\mathcal{E}} + \abs{T_0(\mathcal{E})}}{2} - 1$}},\\
\end{split}
\end{equation}
and that from irreducibility, $x^\star \neq m$. \GW{The last expression in \eqref{equation:T-cardinalities} highlights the fact that $\abs{T(\mathcal{E})}$ is independent of any ordering of elements of $\mathcal{X}$. Note also that the element $(m, x_\star)$ in the definition of $T(\mathcal{E})$ plays no special role, and could be replaced with any other element of $T_0(\mathcal{E}) \cup T_+(\mathcal{E})$.}
We define the sets of symmetric and log-reversible functions (Definition~\ref{definition:revesible-function}) over the graph
$(\mathcal{X}, \mathcal{E})$, respectively by
\begin{equation*}
    \begin{split}
    \mathcal{F} \sym (\mathcal{X}, \mathcal{E}) &\eqdef \set{ h \in  \mathcal{F} (\mathcal{X}, \mathcal{E}) \colon \forall x, x' \in \mathcal{X}, h(x, x') = h(x', x)}, \\
    \mathcal{F} \rev (\mathcal{X}, \mathcal{E}) &\eqdef \set{ h \in  \mathcal{F} (\mathcal{X}, \mathcal{E}) \colon h \text{ is log-reversible }}.
    \end{split}
\end{equation*}
We note that $\mathcal{F} \sym (\mathcal{X}, \mathcal{E})$ is isomorphic to the vector space of symmetric matrices whose entries are null outside of $\mathcal{E}$,
thus $\dim \mathcal{F} \sym (\mathcal{X}, \mathcal{E}) = \abs{T_+(\mathcal{E})} + \abs{T_0(\mathcal{E})}$.
We now show that $\mathcal{F} \rev (\mathcal{X}, \mathcal{E})$ is also a vector space, and that it contains $\mathcal{N}  (\mathcal{X}, \mathcal{E})$ defined at \eqref{definition:anti-shift-functions}. 

\begin{lemma}
\label{lemma:subspace-log-reversible}
The following vector subspace inclusions hold:
\begin{equation*}
    \mathcal{N}  (\mathcal{X}, \mathcal{E}) \stackrel{(i)}{\text{\GW{$\subset$}}} \mathcal{F} \rev (\mathcal{X}, \mathcal{E}) \stackrel{(ii)}{\text{\GW{$\subset$}}} \mathcal{F} (\mathcal{X}, \mathcal{E}).
\end{equation*}
\end{lemma}

\begin{proof}
To verify $(ii)$, we argue that
$\mathcal{F} \rev (\mathcal{X}, \mathcal{E})$ is closed by linear combinations from properties of the sum.
The fact that the null function is trivially reversible concludes this claim. 
\GW{For $(i)$, consider an element $h \in \mathcal{N}(\mathcal{X}, \mathcal{E})$, such that $h(x, x') = f(x') - f(x) + c$. Then $h(x, x') = h(x', x) + 2f(x') - 2f(x)$, and from Corollary~\ref{corollary:easy-log-reversibility}, $h \in \mathcal{F}\rev(\mathcal{X}, \mathcal{E})$, thus the inclusion holds.}
The set is closed by linear combinations by properties of sums again, 
and taking $f = 0, c =0$ is allowed, whence claim $(i)$. 
\end{proof}
\GW{
\begin{remark}
\label{remark:log-reversible-as-direct-sum}
In fact, defining
\begin{equation}
\label{definition:centered-anti-shift-functions}
\begin{split}
\mathcal{N}_0(\mathcal{X}, \mathcal{E}) \eqdef \bigg\{ &h \in \mathcal{F}(\mathcal{X}, \mathcal{E}) \colon \exists f \in \R^\mathcal{X}, \\ &\forall (x, x') \in \mathcal{E}, h(x, x') = f(x') - f(x) \bigg\},
\end{split}
\end{equation}
Corollary~\ref{corollary:easy-log-reversibility} implies that 
$\mathcal{F}\rev(\mathcal{X}, \mathcal{E}) = \mathcal{F} \sym(\mathcal{X}, \mathcal{E}) \oplus \mathcal{N}_0(\mathcal{X}, \mathcal{E})$.
\end{remark}}
It is then possible to further define the quotient space of reversible generator functions
$$\mathcal{G} \rev (\mathcal{X}, \mathcal{E}) \eqdef  \mathcal{F} \rev (\mathcal{X}, \mathcal{E})/\mathcal{N}(\mathcal{X}, \mathcal{E}).$$

\begin{theorem}
\label{theorem:quotient-space-dimension-e-family}
The following statements hold.
\begin{enumerate}
    \item[$(i)$] The set of reversible generators $\mathcal{G}\rev (\mathcal{X}, \mathcal{E})$ can be endowed with a $\abs{T(\mathcal{E})}$-dimensional vector space structure.
    \item[$(ii)$] 
The set $\mathcal{W}\rev(\mathcal{X}, \mathcal{E})$ of irreducible and reversible Markov kernels over $(\mathcal{X}, \mathcal{E})$ forms an e-family of dimension $\dim \mathcal{W}\rev(\mathcal{X}, \mathcal{E}) = \abs{T(\mathcal{E})}$.
\end{enumerate}
\end{theorem}

\begin{proof}
Let $g$ be a log-reversible function over $(\mathcal{X},\mathcal{E})$. From Corollary~\ref{corollary:easy-log-reversibility}, there exists $f \in \R^\mathcal{X}$ such that $g(x, x') = g(x', x) + f(x') - f(x)$, 
or writing $h(x, x') = g(x, x') + \tilde{f}(x) - \tilde{f}(x')$ with $\tilde{f} = f/2$ \GW{(i.e. $h = (g + g \trn)/2$)}, it holds that
$h(x, x') = h(x', x)$, i.e. $h$ is symmetric.
$\mathcal{G} \rev (\mathcal{X},\mathcal{E})$ thus also corresponds to the alternative quotient space
\GW{
\begin{equation}
    \label{eq:aternative-quotient-space}
    \mathcal{G} \rev (\mathcal{X},\mathcal{E}) \cong \mathcal{F} \sym (\mathcal{X},\mathcal{E}) / \R,
\end{equation}}
and as a consequence $\dim \mathcal{G} \rev (\mathcal{X},\mathcal{E}) =  \abs{T_+(\mathcal{E})} + \abs{T_0(\mathcal{E})} - 1 = \abs{T(\mathcal{E})}$.
This concludes the proof of $(i)$.
Let $g \in \mathcal{G} \rev (\mathcal{X}, \mathcal{E})$, and recall the definition \eqref{equation:delta-diffeo} of the diffeomorphism $\Delta$. 
By Theorem~\ref{theorem:characterization-reversible-e-family}, 
$\Delta(\mathcal{G} \rev (\mathcal{X}, \mathcal{E})) \text{ \GW{$\subset$} } \mathcal{W} \rev (\mathcal{X}, \mathcal{E})$.
Conversely, let $P \in \mathcal{W} \rev (\mathcal{X}, \mathcal{E})$. Then by the Kolmogorov criterion (Theorem~\ref{theorem:kolmgorov-criterion}), $\log [P] \in \mathcal{F} \rev (\mathcal{X}, \mathcal{E})$,
and there exist $(g, f, c) \in \mathcal{G} \rev (\mathcal{X}, \mathcal{E}) \times \R^\mathcal{X} \times \R$ such that for any $x, x' \in \mathcal{X}, \log P(x,x') = g(x,x') + f(x') - f(x) + c$ (where $c$ is unique, $f$ is unique up to an additive constant, and both can be recovered from PF theory).
In other words, there exists $g \in \mathcal{G} \rev (\mathcal{X}, \mathcal{E})$, with $P = \stoch \circ \exp [g]$, hence $P \in \Delta(\mathcal{G} \rev (\mathcal{X}, \mathcal{E}))$, proving that
$$\Delta(\mathcal{G} \rev (\mathcal{X}, \mathcal{E})) = \mathcal{W} \rev (\mathcal{X}, \mathcal{E}).$$
Claim $(ii)$ then follows from
\citet[Theorem~2]{nagaoka2017exponential}, 
as discussed at the end of Section~\ref{section:preliminaries}.
\end{proof}

\begin{corollary}
\GW{For the set of positive Markov kernel, $\abs{T_0(\mathcal{E})} = \abs{\mathcal{X}}$ and $\abs{\mathcal{E}} = \abs{\mathcal{X}}^2$, thus $\dim \mathcal{W}\rev(\mathcal{X}, \mathcal{X}^2) = \abs{\mathcal{X}}(\abs{\mathcal{X}}+1)/2 - 1$.}
This is in line with the known number of degrees of freedom of reversible Markov chains \citep{diaconis2006bayesian, pistone2013algebra}.
\end{corollary}

\begin{theorem}
\label{theorem:reversible-basis}
The family of functions $g_{i j} = \delta_i \trn \delta_j + \delta_j \trn \delta_i$,
for $(i,j) \in T(\mathcal{E})$,
forms a basis of $\mathcal{G}\rev (\mathcal{X}, \mathcal{E})$.
\end{theorem}

\begin{proof}
We begin by proving the independence of the family in the quotient space $\mathcal{G} \rev (\mathcal{X}, \mathcal{E})$.
Since
$g_{ij}$ is symmetric in the 
sense that $g_{ij} = g_{ij} \trn$,
it trivially verifies the log-reversibility property, thus belongs to $\mathcal{G} \rev (\mathcal{X}, \mathcal{E})$.
Let now $g \in \mathcal{G} \rev(\mathcal{X}, \mathcal{E})$ be such that
$$g = \sum_{(i, j) \in T(\mathcal{E})} \alpha_{i j} g_{ij},$$
with $\alpha_{ij} \in \R$, for any $(i,j) \in T(\mathcal{E})$, and suppose that $g = 0_{\mathcal{G} \rev (\mathcal{X}, \mathcal{E})}$.
Our first step is to observe that necessarily $g = 0_{\mathcal{F} (\mathcal{X},\mathcal{E} )}$, 
i.e. $g$ must be the null vector in the ambient space.
Let us suppose for contradiction that 
there exist $(f, c) \in (\R^\mathcal{X}, \R)$ such that
$g(x, x') = f(x') - f(x) + c$ and
 either
$c \neq 0$ or $f$ is not constant over $\mathcal{X}$.
Since by definition, $(m, x_\star), (x_\star, m) \not \in T(\mathcal{E})$, 
\begin{equation*}
\begin{split}
0 = g(m, x_\star) = f(m) - f(x_\star) + c, \\    
0 = g(x_\star, m) = f(x_\star) - f(m) + c,
\end{split}
\end{equation*}
therefore summing the latter equalities yields $c = 0$, thus $f$ cannot be constant. 
But then, $g$ is both symmetric and skew-symmetric,
which leads to a contradiction,
and $g = 0_{\mathcal{F}(\mathcal{X}, \mathcal{E})}$.
Since the family $\set{g_{ij} \colon (i,j) \in T(\mathcal{E})}$ is independent in the ambient space $\mathcal{F}(\mathcal{X}, \mathcal{E})$, 
the coefficients $\alpha_{ij}, (i,j) \in T(\mathcal{E})$ must be null, 
and as result, the family is also linearly independent in $\mathcal{G} \rev(\mathcal{X}, \mathcal{E})$.
Finally, since from Theorem~\ref{theorem:quotient-space-dimension-e-family}, $\abs{T(\mathcal{E})} = \dim \mathcal{G} \rev (\mathcal{X}, \mathcal{E})$,
the family is maximally independent, hence constitutes 
a basis of the quotient vector space.
\end{proof}
\GW{
\begin{remark}
An alternative way of showing the linear independence of the family $\set{g_{ij} \colon (i,j) \in T(\mathcal{E})}$ in Theorem~\ref{theorem:reversible-basis} consists in verifying that $(i)$ the family is independent in $\mathcal{F}\sym$, $(ii)$ $\R \ \not \subset \spann \set{g_{ij} \colon (i,j) \in T(\mathcal{E})}$, and then invoking $\eqref{eq:aternative-quotient-space}$.
\end{remark}
}

\subsection{Parametrization of the manifold of reversible kernels}
\label{section:parametrization-reversible-kernels}

Recall that from \citep[Example~1]{nagaoka2017exponential}, in the complete graph case ($\mathcal{E} = \mathcal{X}^2$), we can find an explicit parametrization for $\mathcal{W}(\mathcal{X}, \mathcal{X}^2)$.
Indeed, picking any $x_\star \in \mathcal{X}$, we can easily verify that for the two cases where $x'= x_\star$ and $x' \neq x_\star$, 
\begin{equation*}
\begin{split}
\log P(x, x') = &\sum_{i = 1}^{\abs{\mathcal{X}}} \sum_{\substack{j = 1 \\ j \neq x_\star }}^{\abs{\mathcal{X}}} \log \frac{P(i, j) P(j, x_\star)}{P(i, x_\star)P(x_\star, x_\star)}  \delta_i(x)\delta_j(x') \\ &+ \log P(x, x_\star) - \log P(x', x_\star) + \log P(x_\star, x_\star).
\end{split}
\end{equation*}
In the remainder of this section, we show how to derive a similar parametrization for $\mathcal{W} \rev (\mathcal{X}, \mathcal{E})$. We start by recalling the definition of the expectation 
parameter of an exponential family of kernels. For an e-family $\mathcal{V}_e$, following the notation of Definition~\ref{definition:e-family-parametric}, we define 
\begin{equation*}
\eta_i(\theta) \eqdef Q_\theta[g_i] = \sum_{x, x' \in \mathcal{X}} Q_\theta(x, x') g_i(x,x'),
\end{equation*}
and call $\eta = (\eta_1, \dots, \eta_d)$ the expectation parameter of the family.
We will first derive $\eta$ and later convert to the natural parameter $\theta$ using the following lemma.
\begin{lemma}
\label{lemma:chart-transition-maps}
For a given exponential family, we can express the chart transition maps between the expectation and natural parameters $\theta \circ \eta^{-1}$ and $\eta \circ \theta^{-1}$.
Extending the notation at Lemma~\ref{lemma:q-expectation},
\begin{itemize}
    \item[$(i)$] \begin{equation*}\eta_i(\theta) 
    = Q_\theta[g_i] 
    = \sum_{x, x' \in \mathcal{X}} Q_\theta(x,x') g_i(x,x')
    .\end{equation*}
    \item[$(ii)$] \begin{equation*}
    \begin{split}
        \theta^i(\eta) &= \left(\frac{\partial}{\partial \eta_i} Q_\eta  \right) \left[ \log P_\eta - K \right]   \\
        &= \sum_{x, x' \in \mathcal{X}} \left(\frac{\partial}{\partial \eta_i} Q_\eta(x, x') \right) \left(\log P_\eta(x, x') - K(x, x') \right)
        .
    \end{split}
    \end{equation*}
    In particular, when the carrier kernel verifies $K = 0$, we more simply have $$\theta^i(\eta) = \left(\frac{\partial}{\partial \eta_i} Q_\eta  \right) \left[ \log P_\eta \right].$$
\end{itemize}
\end{lemma}

\begin{proof}
It is well-known that $\eta_i(\theta) = \frac{\partial}{\partial \theta^i} \psi_\theta = Q_\theta [g_i]$ \citep[Lemma~5.1]{hayashi2014information}, \citep[Theorem~4]{nagaoka2017exponential}, 
\citep[(28)]{nakagawa1993converse}, therefore we only need to show $(ii)$.
Let $g_1, g_2, \dots, g_d$ be a collection of independent functions of
$\mathcal{G} (\mathcal{X}, \mathcal{E})$.
Consider the exponential family as in Definition~\ref{definition:e-family-parametric}.
Recall that for two transition kernels $P_1, P_2$ respectively 
irreducible over $(\mathcal{X}, \mathcal{E}_1)$ and  $(\mathcal{X}, \mathcal{E}_2)$,
and with stationary distributions $\pi_1$ and $\pi_2$,
the information divergence of $P_1$ from $P_2$ is given by
\begin{equation}
\label{equation:kl-divergence}
    \kl{P_1}{P_2} = \begin{cases} \sum_{(x,x') \in \mathcal{E}_1} \pi_1(x) P_1(x, x') \log \frac{P_1(x, x')}{P_2(x, x')}, &\text{ when } \mathcal{E}_1 \text{ \GW{$\subset$} } \mathcal{E}_2,
    \\ \infty &\text{ otherwise}.
    \end{cases}
\end{equation}
Writing $P_0$ for $P_\theta$ when $\theta = 0$,
\begin{equation*}
\begin{split}
    & \kl{P_\theta}{P_0} = \sum_{x, x' \in \mathcal{X}} Q_\theta(x,x') \log \frac{P_\theta(x, x')}{P_0(x, x')} \\
    &= \sum_{x, x' \in \mathcal{X}} Q_\theta(x,x') \left[ \sum_{i = 1}^{d}\theta^i g_i(x,x') + R_\theta(x') - R_\theta(x) - \psi_\theta - R_0(x') + R_0(x) + \psi_0 \right] \\
    &= \sum_{i = 1}^{d} \theta^i \eta_i - \psi_\theta + \psi_0,
\end{split}
\end{equation*}
where for the last equality we used $(i)$ of the present lemma and  Lemma~\ref{lemma:q-expectation}-$(iii)$. Moreover, by a direct computation,
\begin{equation*}
    Q_\theta \left[ - \log P_0 \right] = \psi_0 - Q_\theta[K].
\end{equation*}
Thus, the potential function is given by
\begin{equation}
\label{equation:negative-entropy}
\begin{split}
    \varphi(\eta) &\eqdef \sum_{i = 1}^{d} \theta^i \eta_i - \psi_\theta =  Q_\theta \left[\log P_\theta\right] - Q_\theta[K] = Q_\eta \left[ \log P_\eta  \right] - Q_\eta[K].
\end{split}
\end{equation}
By taking the derivative, we recover $\frac{\partial}{\partial \eta_i} \varphi(\eta) = \theta^i(\eta)$ \citep[(17)]{nagaoka2017exponential}.
Moreover, from \eqref{equation:negative-entropy}, we have that 
\begin{equation*}
\begin{split}
\frac{\partial}{\partial \eta_i} \varphi(\eta) &= \frac{\partial}{\partial \eta_i} \left( Q_\eta \left[ \log P_\eta  \right] - Q_\eta[K] \right) \\
&= \left(\frac{\partial}{\partial \eta_i} Q_\eta \right) \left[ \log P_\eta - K \right] + Q_\eta \left[ \frac{\partial \log P_\eta}{\partial \eta_i}  \right]  -  Q_\eta \left[ \frac{\partial K }{\partial \eta_i}  \right] \\
&= \left(\frac{\partial}{\partial \eta_i} Q_\eta \right) \left[ \log P_\eta - K  \right],  \\
\end{split}
\end{equation*}
where for the last equality, we used the fact that $Q_\eta [ \partial K / \partial \eta_i] = 0$, and that from $P_\eta$ being stochastic,
$$\sum_{x,x' \in \mathcal{X}} Q_\eta(x,x') \frac{\partial}{\partial \eta_i} \log P_\eta(x,x') = \sum_{x,x' \in \mathcal{X}} \pi_\eta(x) \frac{\partial}{\partial \eta_i} P_\eta(x,x') = 0.$$
This finishes proving $(ii)$ of the lemma.
\end{proof}

\begin{theorem}
\label{theorem:explicit-parametrization}
Let $P \in \mathcal{W} \rev(\mathcal{X}, \mathcal{E})$, with stationary distribution $\pi$. 
Using the basis $g_{ij} = \delta_i \trn \delta_j  + \delta_j \trn \delta_i$,
we can write $Q$, the edge measure matrix associated with $P$, as
a member of the m-family of reversible kernels,
\begin{equation*}
\begin{split}
    Q &= \frac{g_\star}{2} + \sum_{(i,j) \in T(\mathcal{E})} (g_{ij} - g_\star)\frac{Q(i,j)}{1 + \delta_i(j)},
\end{split}
\end{equation*}
where 
$g_\star = \delta_{m} \trn \delta_{x_\star} + \delta_{x_\star} \trn \delta_{m}$,
and we can write $P$ as a member of the e-family,
\begin{equation*}
\begin{split}
\log P(x, x') &=  \sum_{(i,j) \in T(\mathcal{E})} \frac{1}{2} \log \frac{P(i, j)P(j, i)}{P(m, x_\star )P(x_\star, m )}\left(\frac{g_{ij}(x, x') }{1 + \delta_i(j)}\right) \\
&+ \frac{1}{2} \log \pi(x') - \frac{1}{2} \log \pi(x) + \frac{1}{2} \log P(m, x_\star )P(x_\star, m ),
\end{split}
\end{equation*}
when $(x,x') \in \mathcal{E}$, $P(x,x') = 0$ otherwise, and
where $x_\star = \argmin_{x \in \mathcal{X}} \set{ (m, x) \in \mathcal{E} }$.
\end{theorem}

\begin{proof}
Let us consider the basis
$$g_{i j} = \delta_i \trn \delta_j + \delta_j \trn \delta_i,$$
and taking $K = 0$, we are looking for a parametrization of the type
\begin{equation*}
\begin{split}
    \widetilde{P}_\theta(x, x') & = \exp \left( \sum_{(i,j) \in T(\mathcal{E})} \theta^{ij} g_{ij}(x, x') \right), \\
    P_\theta(x, x') &= \stoch(\widetilde{P}_\theta)(x, x') =  \widetilde{P}_\theta(x, x') \exp \left( R_\theta(x') - R_\theta(x) - \psi_\theta \right),
\end{split}
\end{equation*}
where $\exp \psi_\theta$ and $\exp[R_\theta]$ are respectively 
the PF root and right PF eigenvector of $\widetilde{P}_\theta$.
We first derive a parametrization of the edge measure $Q_\eta$ as a member of an m-family \GW{(following Definition~\ref{definition:m-family-parametric}-$(ii)$} with respect to the expectation parameter $\eta$).
For $(i, j) \in \mathcal{X}^2$, by Lemma~\ref{lemma:chart-transition-maps}-$(i)$,
$$\eta_{ij} = Q_\eta[g_{ij}] = Q_\eta(i,j) + Q_\eta(j, i) = 2Q_\eta( i,j),$$
\GW{and thus, from symmetry of $Q_\eta$ and since $Q_\eta \in \mathcal{P}(\mathcal{X}^2)$,}
\begin{equation*}
\begin{split}
Q_\eta(i, j) = 
    \begin{cases}
    0 & \text{ when } (i,j) \not \in \mathcal{E} \\
    \eta_{ij} /2 & \text{ when } (i, j) \in T(\mathcal{E}), i \neq j \\
    \eta_{ji}/2 & \text{ when } (j, i) \in T(\mathcal{E}), i \neq j \\
    \eta_{ii}/2 & \text{ when } (i, i) \in T_0(\mathcal{E}) \\
    \text{ \GWS{$\frac{1}{2}\left(1 - \sum_{(i,j) \in T(\mathcal{E})}\frac{\eta_{ij}}{1 + \delta_i(j)} \right)$}}  & \text{ when } (i, j) \in \set{ (m, x_\star), (x_\star, m) },
    \end{cases}
\end{split}
\end{equation*}
and more compactly, for $(x, x') \in \mathcal{X}^2$,
\begin{equation*}
\begin{split}
    Q_\eta &= \frac{g_\star}{2} + \sum_{(i,j) \in T(\mathcal{E})} \frac{g_{ij} - g_\star}{2(1 + \delta_i(j))}\eta_{ij},
\end{split}
\end{equation*}
where  
$g_\star$ is defined as in the statement of the theorem.
We differentiate by $\eta_{i j}$ for $(i,j) \in T(\mathcal{E})$, to obtain
\begin{equation*}
    \frac{\partial Q_\eta}{\partial \eta_{ij}} = \frac{ g_{ij} - g_\star}{2(1 + \delta_{i}(j))}.
\end{equation*}
Invoking $(ii)$ of Lemma~\ref{lemma:chart-transition-maps}, we convert the expectation parametrization to a natural one, 
\begin{equation*}
    \theta^{ij}(\eta) = \sum_{(x, x') \in \mathcal{X}} \left( \frac{\partial }{\partial \eta_{ij}}Q_\eta(x,x') \right) \log P(x, x'),
\end{equation*}
so that
\begin{equation*}
\begin{split}
    \theta^{ij}(\eta) &= \frac{1}{1 + \delta_{i}(j)} \sum_{(x, x') \in \mathcal{X}} \left( \frac{ g_{ij}(x,x') -  g_\star(x, x') }{2}\right) \log P(x, x') \\
    &= \frac{1}{2(1 + \delta_{i}(j))} \log \frac{P(i, j)P(j, i)}{P(m, x_\star)P(x_\star, m)}. \\
\end{split}
\end{equation*}
Notice that $\widetilde{P}_\theta = \widetilde{P}_\theta \trn$,
hence the right and left PF eigenvector are identical, i.e.
$R_\theta = L_\theta$ and as is known (see \eqref{equation:parametric-stationary}), 
the stationary distribution is given by $\pi_\theta = \exp[2 R_\theta]/\sum_{x \in \mathcal{X}} \exp(2 R_\theta(x))$.
In fact, we can easily verify that the right PF 
eigenvector is given by $\exp[R_\theta] = \sqrt{\pi}$, 
and that the PF root is 
$$\exp \psi_\theta = (P(m , x_\star) P(x_\star, m))^{-1/2}.$$
Indeed, letting $x \in \mathcal{X}$, from detailed balance of $P$, we have
\begin{equation*}
\begin{split}
\sum_{x' \in \mathcal{X}} \widetilde{P}_\theta(x,x') \sqrt{\pi(x')} &= \sum_{x' \in \mathcal{X}} \sqrt{\frac{P(x , x') P(x', x)}{P(m , x_\star) P(x_\star, m)}} = \frac{\sqrt{\pi(x)}}{\sqrt{P(m , x_\star) P(x_\star, m)}}.
\end{split}
\end{equation*}
\end{proof}

\subsection{The doubly autoparallel submanifold of reversible kernels}
\label{section:doubly-autoparallel}
Recall that we can view $\mathcal{W} = \mathcal{W}(\mathcal{X}, \mathcal{E})$ as a smooth manifold of dimension $d = \dim \mathcal{W} = \abs{\mathcal{E}} - \abs{\mathcal{X}}$.
For each $P \in \mathcal{W}$,
we can then consider the tangent plane $T_P$ at $P$, endowed with a $d$-dimensional vector space structure.
Together with the manifold, we define an information geometric structure consisting of a Riemannian metric, called
the \emph{Fisher information metric} $\mathfrak{g}$,
and a pair of \emph{torsion-free affine connections} $\nabla^{(e)}$ and $\nabla^{(m)}$ respectively called \emph{e-connection} and \emph{m-connection}, that are dual with respect to $\mathfrak{g}$,
i.e. for any vector fields $X, Y, Z \in \Gamma(T\mathcal{W})$,
\begin{equation*}
\begin{split}
X \mathfrak{g}(Y, Z) = \mathfrak{g}(\nabla^{(e)}_X Y, Z) + \mathfrak{g}( Y, \nabla^{(m)}_X Z),
\end{split}
\end{equation*}
where $\Gamma(T \mathcal{W})$ is the set of all sections over the tangent bundle.
We now review an explicit construction for $\mathfrak{g}, \nabla^{(m)}, \nabla^{(e)}$.

\paragraph{Construction in the natural chart map.}

Consider a parametric family $\mathcal{V} = \set{P_\theta \colon \theta \in \Theta}$ with $\Theta$ open subset of $\R^d$.
For any $n \in \N$, we define the path measure $Q_\theta^{(n)} \in \mathcal{P}(\mathcal{X}^n)$ induced from the kernel $P_\theta$.
$$Q_\theta^{(n)}(x_1, x_2, \dots, x_n) = \pi_\theta(x_1) \prod_{t = 1}^{n - 1} P_\theta(x_t, x_{t+1}).$$
\citet{nagaoka2017exponential} defines the Fisher metric as
\begin{equation*}
\begin{split}
\mathfrak{g}_{ij}(\theta) &\eqdef \sum_{(x, x') \in \mathcal{E}} Q_\theta(x, x') \partial_i \log P_\theta(x, x') \partial_j \log P_\theta(x, x'), \\
&= \sum_{(x, x') \in \mathcal{E}} \partial_i \log P_\theta(x, x') \partial_j Q_\theta(x, x'), \\
&= \lim_{n \to \infty} \frac{1}{n} \mathfrak{g}_{ij}^{ n}(\theta),
\end{split}
\end{equation*}
and the dual affine e/m-connections of $\set{P_\theta \colon \theta \in \Theta}$ by their Christoffel symbols,
\begin{equation*}
\begin{split}
\Gamma^{(e)}_{ij, k}(\theta) &\eqdef \sum_{(x, x') \in \mathcal{E}} \partial_i \partial_j \log P_\theta(x, x') \partial_k Q_\theta(x,  x') = \lim_{n \to \infty} \frac{1}{n} \Gamma^{(e), n}_{ij, k}(\theta), \\
\Gamma^{(m)}_{ij, k}(\theta) &\eqdef \sum_{(x, x') \in \mathcal{E}} \partial_i \partial_j  Q_\theta(x, x') \partial_k \log P_\theta(x,  x') = \lim_{n \to \infty} \frac{1}{n} \Gamma^{(m), n}_{ij, k}(\theta), \\
\end{split}
\end{equation*}
where $\mathfrak{g}_{ij}^{n}(\theta), \Gamma^{(e), n}_{ij, k}(\theta), \Gamma^{(m), n}_{ij, k}(\theta)$ are the Fisher metric, and Christoffel symbols of the e/m-connections that pertain to the distribution family $\set{Q_\theta^{(n) }}_{\theta \in \Theta}$.

\paragraph{Autoparallelity.}

Connections allow us to talk about covariant derivatives and parallelity of vectors fields.

\begin{definition}
A submanifold $\mathcal{V}$  is called  autoparallel in $\mathcal{W}$ with respect to a connection $\nabla$, when for any vector fields $\forall X, Y \in \Gamma(T \mathcal{V})$, it holds that
$$\nabla_X Y \in \Gamma(T \mathcal{V}).$$
\end{definition}
A submanifold $\mathcal{V}$ of $\mathcal{W}$ is then an e-family (resp. m-family) if and only if it is autoparallel with respect to $\nabla^{(e)}$ (resp. $\nabla^{(m)}$) \citep[Theorem~6]{nagaoka2017exponential}.
As the manifold of reversible kernels is both an e-family and an m-family, it is called doubly autoparallel \citep[Definition~1]{ohara2016doubly}.

\begin{theorem}
\label{theorem:doubly-autoparallel-submanifold}
The manifold $\mathcal{W}\rev(\mathcal{X}, \mathcal{E})$ of irreducible and reversible Markov chains over $(\mathcal{X}, \mathcal{E})$ is a doubly autoparallel submanifold in $\mathcal{W}(\mathcal{X}, \mathcal{E})$ with
dimension 
\GW{$$\dim \mathcal{W}\rev(\mathcal{X}, \mathcal{E}) = \frac{\abs{\mathcal{E}} + \abs{T_0(\mathcal{E})}}{2} - 1,$$
where $T_0(\mathcal{E}) = \set{ (x,x') \in \mathcal{E} \colon x = x' }$.
}
\end{theorem}

\begin{proof}
The set of reversible Markov chains is an e-family (Theorem~\ref{theorem:quotient-space-dimension-e-family}), and an m-family (Theorem~\ref{theorem:explicit-parametrization}).
\end{proof}

\subsection{Reversible geodesics}
\label{section:reversible-geodesics}
In this section, we let two
irreducible reversible kernels $P_0$ and $P_1$ over $(\mathcal{X}, \mathcal{E})$,
and discuss the geodesics that connect them with respect to $\nabla^{(e)}$ and $\nabla^{(m)}$.
Although already guaranteed (see for example \citet[Proposition~1]{ohara2016doubly}), we offer alternative elementary proofs that any kernel lying on either e/m-geodesic is irreducible and reversible.

\paragraph{m-geodesics.}

By irreducibility, there exist unique $Q_0, Q_1, \in \mathcal{Q}(\mathcal{X}, \mathcal{E})$ corresponding to $P_0, P_1$. Moreover, by reversibility $Q_0$ and $Q_1$ are symmetric.
We let
$$G_m(P_0, P_1) \eqdef \set{P_\xi \colon Q_\xi = \xi Q_1 + (1 - \xi) Q_0 \colon \xi \in [0, 1]},$$
be the m-geodesic (auto-parallel curve with respect to the m-connection) connecting $P_0$ and $P_1$. Then $G_m(P_0, P_1)$ forms an m-family of dimension $1$. 
For any $\xi \in [0,1]$, the matrix $Q_\xi$ is symmetric as convex combination of two symmetric matrices. $Q_\xi$ takes value $0$ exactly when $Q_0, Q_1$, i.e. $P_0, P_1$ take value $0$. 
Furthermore, writing $\pi_0$ (resp. $\pi_1$) the unique stationary distribution of $P_0$ (resp. $P_1$),
\begin{equation*}
    \begin{split}
        \sum_{x'} Q_\xi(x, x') &= \xi \pi_1(x) + (1 - \xi) \pi_0(x), \\
\sum_{x} Q_\xi(x, x') &= \xi \pi_1(x') + (1 - \xi) \pi_0(x'),
    \end{split}
\end{equation*}
thus $Q_\xi$ always defines a proper associated stochastic irreducible stochastic $P_\xi$.

\paragraph{e-geodesics.}

We consider the auto-parallel curve with respect to the e-connection that connect $P_0$ and $P_1$,
$$G_e(P_0, P_1) \eqdef \set{ P_0(x, x') \exp \left( \theta \ln \frac{P_{1}(x, x')}{P_0(x, x')} + R_\theta(x') - R_\theta(x) - \psi_\theta \right) : \theta \in [0,1] }.$$
The set $G_e(P_0, P_1)$ forms an e-family of dimension 1.
Indeed, from Theorem~\ref{theorem:characterization-reversible-e-family}, 
and since $P_0$ and $P_1$ are reversible by hypothesis, 
it suffices to verify that $(x, x') \mapsto P_1(x, x') / P_0(x, x')$  is a reversible function over $(\mathcal{X}, \mathcal{E})$. This follows from a simple application of the Kolmogorov criterion (Theorem~\ref{theorem:kolmgorov-criterion}).

\section{Reversible information projections}
\label{section:information-projections}
Reversible Markov kernels, as self-adjoint linear operators, 
enjoy a set of powerful yet brittle spectral properties.
The eigenvalues are real, 
the second largest in magnitude controls 
the time to stationarity of the Markov process \citep[Chapter~12]{levin2009markov},
and all are stable under perturbation and estimation \citep{hsu2019}.
However, any deviation from reversibility carries steep consequences, 
as the spectrum can suddenly become complex, and partially loses control over the
mixing time. Furthermore, eigenvalue perturbation results 
that were dimensionless \citep[Corollary~4.10 ~(Weyl's~inequality)]{stewart1990matrix} now come at a cost possibly 
exponential in the dimension \citep[Theorem~1.4 ~(Ostrowski-Elsner)]{stewart1990matrix}.
For some irreducible $P$ with stationary distribution $\pi$, 
it is therefore interesting to find the \emph{closest} 
representative that is reversible, 
so as to enable Hilbert space techniques.
Computing the closest reversible transition kernel with respect to a norm 
induced from an inner product was considered in \citet{nielsen2015computing},
who showed that the problem reduces to 
solving a convex minimization problem with a unique solution.

In this section, we examine this problem under a different notion of distance.
We consider information projections onto the reversible family of transition kernels 
$\mathcal{W}\rev(\mathcal{X}, \mathcal{E})$, for some symmetric edge set $\mathcal{E}$.
We define the m-projection and the e-projection of $P$ onto the set of reversible transition kernels $\mathcal{W} \rev (\mathcal{X}, \mathcal{E})$ respectively as
\begin{equation*}
\begin{split}
    P_m \eqdef \argmin_{\bar{P} \in \mathcal{W} \rev (\mathcal{X}, \mathcal{E})} \kl{P}{\bar{P}}, \qquad
    P_e \eqdef \argmin_{\bar{P} \in \mathcal{W} \rev (\mathcal{X}, \mathcal{E})} \kl{\bar{P}}{P},
\end{split}
\end{equation*}
where $\kl{\cdot}{\cdot}$ is the informational divergence, that was defined at \eqref{equation:kl-divergence}.
These two generally distinct projections ($D$ is not symmetric in its arguments) 
correspond to the closest reversible chains 
when considering information divergence as a measure of distance.
Under a careful choice of the connection graph of the reversible family,
we derive closed-form expressions for $P_m$ and $P_e$,
along with Pythagorean identities, as illustrated in Figure~\ref{figure:information-projection}.
\begin{theorem}
\label{theorem:e-m-projections}
Let $P$ be irreducible over $(\mathcal{X}, \mathcal{E})$. 

\paragraph{m-projection.}

The m-projection $P_m$ of $P$ onto $\mathcal{W} \rev (\mathcal{X}, \mathcal{E} \cup \mathcal{E}^\star)$ is given by
\begin{equation*}
\begin{split}
P_m &= \frac{P + P^\star}{2}.
\end{split}
\end{equation*}
Moreover, for any $\bar{P} \in \mathcal{W} \rev (\mathcal{X}, \mathcal{E} \cup \mathcal{E}^\star)$, 
$P_m$ satisfies the following Pythagorean identity.
\begin{equation*}
    \kl{P}{\bar{P}} = \kl{P}{P_m} + \kl{P_m}{\bar{P}}.
\end{equation*}

\paragraph{e-projection.}
When $\mathcal{E} \cap \mathcal{E}^\star$ is a strongly connected directed graph,
the e-projection $P_e$ of $P$ onto $\mathcal{W} \rev (\mathcal{X}, \mathcal{E} \cap \mathcal{E}^\star)$ is given by
\begin{equation*}
\begin{split}
P_e &= \stoch(\widetilde{P}_e), \qquad\text{ with } \widetilde{P}_e(x,x') = \sqrt{P(x,x' )P^\star(x,x')},
\end{split}
\end{equation*}
and where $\stoch$ is the stochastic rescaling mapping defined at \eqref{equation:stochastic-version}.
Moreover, for any $\bar{P} \in \mathcal{W} \rev (\mathcal{X}, \mathcal{E} \cap \mathcal{E}^\star)$, 
$P_e$ satisfies the following Pythagorean identity.
\begin{equation*}
\kl{\bar{P}}{P} = \kl{\bar{P}}{P_e} + \kl{P_e}{P}.
\end{equation*}
\end{theorem}

\begin{proof}
Our first order of business is to show that $P_m$ and $P_e$ belong 
respectively to $\mathcal{W} \rev (\mathcal{X}, \mathcal{E} \cup \mathcal{E}^\star)$ 
and $\mathcal{W} \rev (\mathcal{X}, \mathcal{E} \cap \mathcal{E}^\star)$.
It is easy to see that $P_m(x, x') > 0$ exactly 
when $(x, x')$ or $(x', x)$ belongs to $\mathcal{E}$,
hence $P_m \in \mathcal{W} \rev (\mathcal{X}, \mathcal{E} \cup \mathcal{E} ^\star)$,
and that $P_e(x, x') > 0$ whenever $(x, x')$ 
belongs to both $\mathcal{E}$ and $\mathcal{E}^\star$.
Moreover, since the time-reversal operation preserves the 
stationary distribution of an irreducible chain, $P_m$ has
the same stationary distribution $\pi_m = \pi$, and a
straightforward computation shows that $P_m$ satisfies 
the detailed balance equation. 
To prove reversibility of $P_e$, we rewrite
\GW{
\begin{equation*}
\begin{split}
\log P_e(x, x') &= \frac{1}{2}\log P(x,x')P(x',x) - \log \rho(\widetilde{P}_e) \\
&+ \log \left( \sqrt{\pi(x')} v_e(x') \right) - \log \left( \sqrt{\pi(x)} v_e(x) \right).
\end{split}
\end{equation*}
From Corollary~\ref{corollary:easy-log-reversibility}, $\log[P_e] \in \mathcal{F} \rev(\mathcal{X},\mathcal{E})$, thus $P_e \in \mathcal{W} \rev(\mathcal{X},\mathcal{E})$.
}

To prove optimality of $P_m$, it suffices to verify the following Pythagorean identity
\begin{equation*}
    \kl{P}{\bar{P}} = \kl{P}{P_m} + \kl{P_m}{\bar{P}}.
\end{equation*}
Writing $Q_m = \diag(\pi) P_m$,
notice that $P_m = (P + P^\star)/2$ is equivalent to 
$Q_m = (Q + Q^\star)/2$. We then have
\begin{equation*}
\begin{split}
    &\kl{P}{P_m} + \kl{P_m}{\bar{P}} - \kl{P}{\bar{P}} \\ 
    &= \sum_{x, x' \in \mathcal{X}} \bigg( Q(x,x') \log \frac{P(x,x')}{P_m(x, x')} + Q_m(x,x') \log \frac{P_m(x,x')}{\bar{P}(x,x')} \\
    &\qquad - Q(x,x') \log \frac{P(x,x')}{\bar{P}(x,x')} \bigg) \\
    &= \sum_{x, x' \in \mathcal{X}} \left( Q_m(x,x') - Q(x,x') \right) \log \frac{P_m(x,x')}{\bar{P}(x,x')}  \\
    &= \sum_{x, x' \in \mathcal{X}} \left(\frac{ Q^\star(x,x') - Q(x,x')}{2}\right) \log \frac{P_m(x,x')}{\bar{P}(x,x')} \\
		&= \frac{1}{2} Q^\star \left[ \log (P_m / \bar{P}) \right] - \frac{1}{2} Q\left[ \log (P_m / \bar{P}) \right] = 0, \\
\end{split}
\end{equation*}
where the last equality stems from $(i)$ of Lemma~\ref{lemma:q-expectation} 
and reversibility of $P_m$ and $\bar{P}$.
Similarly, to prove optimality of $P_e$, it suffices to verify that
\begin{equation*}
\kl{\bar{P}}{P} = \kl{\bar{P}}{P_e} + \kl{P_e}{P}.
\end{equation*}
By reorganizing terms
\begin{equation*}
\begin{split}
    \kl{\bar{P}}{P_e} + \kl{P_e}{P} - \kl{\bar{P}}{P}
		= \bar{Q}\left[\log (P / P_e) \right] - Q_e\left[ \log (P / P_e) \right].  \\
\end{split}
\end{equation*}
From the definition of $P_e(x,x')$,
\begin{equation*}
\begin{split}
		\log \frac{P(x,x')}{P_e(x,x')} = \frac{1}{2}\log \frac{P(x,x')}{P(x',x)} + \frac{1}{2}\log \frac{\pi(x)}{\pi(x')} + \log \frac{v_e(x)}{v_e(x')} + \log \rho(\widetilde{P}_e). \\
\end{split}
\end{equation*}
The first three terms being skew-symmetric, reversibility of $\bar{P}$ and $(ii)$ 
of Lemma~\ref{lemma:q-expectation} yield that
\begin{equation*}
\begin{split}
		\bar{Q} \left[ \log (P / P_e) \right] = \log \rho(\widetilde{P}_e). \\
\end{split}
\end{equation*}
By a similar argument, $Q_e \left[ \log ( P /P_e) \right] = \log \rho(\widetilde{P}_e)$,
which concludes the proof.
\end{proof}

In other words, the m-projection is given by the natural \emph{additive reversiblization} \citep[(2.4)]{fill1991eigenvalue} of $P$, while the e-projection is achieved by some newly defined \emph{exponential reversiblization} of $P$.

The difference between the m-projection and the e-projection is illustrated in the following example.
\begin{example}
Let us consider the family of biased lazy random walks \GW{$P_\theta = P_{(\theta_1, \theta_2)}$}, given in Example \ref{example:lazy-random-walk}.
Note that $\mathcal{E}=\mathcal{E}^\star$.
The m-projection $P_m$ of $P_\theta$ onto $\mathcal{W} \rev (\mathcal{X}, \mathcal{E})$ is 
the unbiased lazy random walk given by \GW{$P_m = P_{(\theta', 0)}$ with $\theta' = \theta_1 - \log \cosh \theta_2$, i.e. }
\begin{align*}
P_m(x,x) = \frac{e^{\theta_1}}{e^{\theta_1}+e^{\theta_2}+e^{-\theta_2}},
\end{align*}
\begin{align*}
P_m(x,x+1) = P_m(x+1,x)= \frac{e^{\theta_2}+e^{-\theta_2}}{2(e^{\theta_1}+e^{\theta_2}+e^{-\theta_2})}.
\end{align*}
On the other hand, the e-projection $P_e$ of $P_\theta$ onto $\mathcal{W} \rev (\mathcal{X}, \mathcal{E})$ is 
the unbiased lazy random walk given by \GW{$P_e = P_{(\theta_1, 0)}$, i.e.}
\begin{align*}
P_e(x,x) = \frac{e^{\theta_1}}{e^{\theta_1}+2},~~~
P_e(x,x+1) = P_e(x+1,x)= \frac{1}{e^{\theta_1}+2}.
\end{align*}
\end{example}

\begin{remark}
We observe that, although the m-projection preserves the stationary distribution,
this is not true for $P_e$, which
exhibits a stationary distribution $\pi_e$ generally different from $\pi$. 
Furthermore, while the solution for the m-projection is always 
properly defined by taking union of the edge sets, our expression 
for the e-projection requires additional constraints on the
connection graph of $P$. Indeed, taking the intersection $\mathcal{E} \cap \mathcal{E}^\star$, 
we always obtain a symmetric set, but can lose strong connectedness. 
We note but do not pursue the fact that reversibility 
can be  defined for the less well-behaved set of reducible chains.
In this case, $\pi$ need not be unique, or could take null values, and the kernel could have a complex spectrum.
\end{remark}

Finally, we show that for 
any irreducible $P$,
both its reversible projections $P_m$ and $P_e$ are equidistant from $P$ 
 and its time-reversal $P^\star$ (see also Figure~\ref{figure:information-projection}). 

\begin{proposition}[Bisection property]
\label{proposition:divergence-symmetry}
Let $P$ irreducible, and let $P_m$ (resp. $P_e$) the m-projection
(resp. e-projection) of $P$ onto $\mathcal{W}\rev(\mathcal{X}, \mathcal{E})$.
\begin{equation*}
\begin{split}
    \kl{P}{P_m} = \kl{P^\star}{P_m}, \qquad \kl{P_e}{P} = \kl{P_e}{P^\star}. \\
\end{split}
\end{equation*}

\end{proposition}

\begin{proof}
\GW{For $P_1$ irreducible over $(\mathcal{X}, \mathcal{E}_1)$ and $P_2$ irreducible over $(\mathcal{X}, \mathcal{E}_2)$, it is easy to see that
$$\mathcal{E}_1 \text{ \GW{$\subset$} } \mathcal{E}_2 \implies \kl{P_1}{P_2} = \kl{P_1^\star}{P_2^\star}.$$
Then take $P_2 = P_m$ for the first equality, and $P_1 = P_e$ for the second.
}
\end{proof}

\tikzset{every picture/.style={line width=0.75pt}} %

\begin{figure}
\centering

\begin{tikzpicture}[x=0.75pt,y=0.75pt,yscale=-1,xscale=1]

\draw   (192,136) .. controls (214.4,138.1) and (277.2,168.3) .. (413.8,88.2) .. controls (550.4,8.1) and (578.4,168.1) .. (522.4,224.1) .. controls (466.4,280.1) and (207.4,281.1) .. (172.4,259.1) .. controls (137.4,237.1) and (169.6,133.9) .. (192,136) -- cycle ;
\draw [color={rgb, 255:red, 208; green, 2; blue, 27 }  ,draw opacity=1 ]   (332.8,10.2) .. controls (308.52,26.02) and (296.52,88.47) .. (305.66,157.71) ;
\draw [shift={(305.8,158.75)}, rotate = 262.31] [color={rgb, 255:red, 208; green, 2; blue, 27 }  ,draw opacity=1 ][line width=0.75]    (10.93,-3.29) .. controls (6.95,-1.4) and (3.31,-0.3) .. (0,0) .. controls (3.31,0.3) and (6.95,1.4) .. (10.93,3.29)   ;
\draw [color={rgb, 255:red, 208; green, 2; blue, 27 }  ,draw opacity=1 ]   (386.8,142.2) .. controls (382.42,120.21) and (389.33,57.73) .. (340,10.9) ;
\draw [shift={(339.25,10.2)}, rotate = 403.08000000000004] [color={rgb, 255:red, 208; green, 2; blue, 27 }  ,draw opacity=1 ][line width=0.75]    (10.93,-3.29) .. controls (6.95,-1.4) and (3.31,-0.3) .. (0,0) .. controls (3.31,0.3) and (6.95,1.4) .. (10.93,3.29)   ;
\draw  [fill={rgb, 255:red, 255; green, 255; blue, 255 }  ,fill opacity=1 ] (326.35,10.2) .. controls (326.35,6.64) and (329.24,3.75) .. (332.8,3.75) .. controls (336.36,3.75) and (339.25,6.64) .. (339.25,10.2) .. controls (339.25,13.76) and (336.36,16.65) .. (332.8,16.65) .. controls (329.24,16.65) and (326.35,13.76) .. (326.35,10.2) -- cycle ;
\draw  [color={rgb, 255:red, 208; green, 2; blue, 27 }  ,draw opacity=1 ][fill={rgb, 255:red, 255; green, 255; blue, 255 }  ,fill opacity=1 ] (380.35,142.2) .. controls (380.35,138.64) and (383.24,135.75) .. (386.8,135.75) .. controls (390.36,135.75) and (393.25,138.64) .. (393.25,142.2) .. controls (393.25,145.76) and (390.36,148.65) .. (386.8,148.65) .. controls (383.24,148.65) and (380.35,145.76) .. (380.35,142.2) -- cycle ;
\draw [dotted, color={rgb, 255:red, 208; green, 2; blue, 27 }  ,draw opacity=1 ]   (354.4,242.1) .. controls (338.56,231.21) and (319.78,221.3) .. (306.21,173.12) ;
\draw [shift={(305.8,171.65)}, rotate = 434.62] [color={rgb, 255:red, 208; green, 2; blue, 27 }  ,draw opacity=1 ][line width=0.75]    (10.93,-3.29) .. controls (6.95,-1.4) and (3.31,-0.3) .. (0,0) .. controls (3.31,0.3) and (6.95,1.4) .. (10.93,3.29)   ;
\draw [dotted, color={rgb, 255:red, 208; green, 2; blue, 27 }  ,draw opacity=1 ]   (386.8,148.65) .. controls (388.37,179.47) and (377.83,212.74) .. (367.06,240.41) ;
\draw [shift={(366.4,242.1)}, rotate = 291.45] [color={rgb, 255:red, 208; green, 2; blue, 27 }  ,draw opacity=1 ][line width=0.75]    (10.93,-3.29) .. controls (6.95,-1.4) and (3.31,-0.3) .. (0,0) .. controls (3.31,0.3) and (6.95,1.4) .. (10.93,3.29)   ;
\draw  [color={rgb, 255:red, 208; green, 2; blue, 27 }  ,draw opacity=1 ][fill={rgb, 255:red, 255; green, 255; blue, 255 }  ,fill opacity=1 ] (299.35,165.2) .. controls (299.35,161.64) and (302.24,158.75) .. (305.8,158.75) .. controls (309.36,158.75) and (312.25,161.64) .. (312.25,165.2) .. controls (312.25,168.76) and (309.36,171.65) .. (305.8,171.65) .. controls (302.24,171.65) and (299.35,168.76) .. (299.35,165.2) -- cycle ;
\draw  [fill={rgb, 255:red, 255; green, 255; blue, 255 }  ,fill opacity=1 ] (353.95,247.1) .. controls (353.95,243.54) and (356.84,240.65) .. (360.4,240.65) .. controls (363.96,240.65) and (366.85,243.54) .. (366.85,247.1) .. controls (366.85,250.66) and (363.96,253.55) .. (360.4,253.55) .. controls (356.84,253.55) and (353.95,250.66) .. (353.95,247.1) -- cycle ;
\draw    (495.4,131.1) .. controls (495.4,89.31) and (415.21,16.83) .. (340.38,10.29) ;
\draw [shift={(339.25,10.2)}, rotate = 364.49] [color={rgb, 255:red, 0; green, 0; blue, 0 }  ][line width=0.75]    (10.93,-3.29) .. controls (6.95,-1.4) and (3.31,-0.3) .. (0,0) .. controls (3.31,0.3) and (6.95,1.4) .. (10.93,3.29)   ;
\draw  [fill={rgb, 255:red, 255; green, 255; blue, 255 }  ,fill opacity=1 ] (488.95,131.1) .. controls (488.95,127.54) and (491.84,124.65) .. (495.4,124.65) .. controls (498.96,124.65) and (501.85,127.54) .. (501.85,131.1) .. controls (501.85,134.66) and (498.96,137.55) .. (495.4,137.55) .. controls (491.84,137.55) and (488.95,134.66) .. (488.95,131.1) -- cycle ;
\draw    (326.8,13.2) .. controls (292,40.05) and (200.32,80.69) .. (209.26,174.68) ;
\draw [shift={(209.4,176.1)}, rotate = 263.99] [color={rgb, 255:red, 0; green, 0; blue, 0 }  ][line width=0.75]    (10.93,-3.29) .. controls (6.95,-1.4) and (3.31,-0.3) .. (0,0) .. controls (3.31,0.3) and (6.95,1.4) .. (10.93,3.29)   ;
\draw  [fill={rgb, 255:red, 255; green, 255; blue, 255 }  ,fill opacity=1 ] (202.95,182.55) .. controls (202.95,178.99) and (205.84,176.1) .. (209.4,176.1) .. controls (212.96,176.1) and (215.85,178.99) .. (215.85,182.55) .. controls (215.85,186.11) and (212.96,189) .. (209.4,189) .. controls (205.84,189) and (202.95,186.11) .. (202.95,182.55) -- cycle ;
\draw    (299.35,165.2) -- (217.81,182.14) ;
\draw [shift={(215.85,182.55)}, rotate = 348.26] [color={rgb, 255:red, 0; green, 0; blue, 0 }  ][line width=0.75]    (10.93,-3.29) .. controls (6.95,-1.4) and (3.31,-0.3) .. (0,0) .. controls (3.31,0.3) and (6.95,1.4) .. (10.93,3.29)   ;
\draw    (488.95,131.1) -- (395.24,141.97) ;
\draw [shift={(393.25,142.2)}, rotate = 353.38] [color={rgb, 255:red, 0; green, 0; blue, 0 }  ][line width=0.75]    (10.93,-3.29) .. controls (6.95,-1.4) and (3.31,-0.3) .. (0,0) .. controls (3.31,0.3) and (6.95,1.4) .. (10.93,3.29)   ;

\draw  [color={rgb, 255:red, 65; green, 117; blue, 5 }  ,draw opacity=1 ] (384.62,118.04) -- (406.98,115.7) -- (409.59,140.63) ;
\draw [color={rgb, 255:red, 208; green, 2; blue, 27 }  ,draw opacity=1 ]   (293.98,103.93) -- (311.18,104.02) ;
\draw [color={rgb, 255:red, 208; green, 2; blue, 27 }  ,draw opacity=1 ]   (294.22,98.28) -- (311.42,98.37) ;
\draw [color={rgb, 255:red, 208; green, 2; blue, 27 }  ,draw opacity=1 ]   (333.78,213.96) -- (321.73,226.23) ;
\draw [color={rgb, 255:red, 208; green, 2; blue, 27 }  ,draw opacity=1 ]   (337.67,218.07) -- (325.62,230.34) ;
\draw [color={rgb, 255:red, 208; green, 2; blue, 27 }  ,draw opacity=1 ]   (362.38,57.15) -- (379.37,54.53) ;
\draw [color={rgb, 255:red, 208; green, 2; blue, 27 }  ,draw opacity=1 ]   (376.07,176.68) -- (392.7,181.03) ;
\draw  [color={rgb, 255:red, 65; green, 117; blue, 5 }  ,draw opacity=1 ] (280.47,168.53) -- (278.01,146.19) -- (304.01,143.32) ;

\draw (314,152.4) node [anchor=north west][inner sep=0.75pt]  [color={rgb, 255:red, 208; green, 2; blue, 27 }  ,opacity=1 ]  {$P_{e}$};
\draw (361.8,139.6) node [anchor=north west][inner sep=0.75pt]  [color={rgb, 255:red, 208; green, 2; blue, 27 }  ,opacity=1 ]  {$P_{m}$};
\draw (307,-0.6) node [anchor=north west][inner sep=0.75pt]    {$P$};
\draw (371,238.4) node [anchor=north west][inner sep=0.75pt]    {$P^{\star }$};
\draw (506,119.4) node [anchor=north west][inner sep=0.75pt]    {$\bar{P}$};
\draw (208,242.4) node [anchor=north west][inner sep=0.75pt]    {$\mathcal{W} \rev (\mathcal{X}, \mathcal{X}^2)$};
\draw (185,172.4) node [anchor=north west][inner sep=0.75pt]    {$\bar{P}$};
\draw (185,53.4) node [anchor=north west][inner sep=0.75pt]    {$\kl{\bar{P}}{P}$};
\draw (430,23.4) node [anchor=north west][inner sep=0.75pt]    {$\kl{P}{\bar{P}}$};
\draw (420,138.4) node [anchor=north west][inner sep=0.75pt]    {$\kl{P_m}{\bar{P}}$};
\draw (238,177.4) node [anchor=north west][inner sep=0.75pt]    {$\kl{\bar{P}}{P_e}$};
\draw (305,74.4) node [anchor=north west][inner sep=0.75pt]  [color={rgb, 255:red, 208; green, 2; blue, 27 }  ,opacity=1 ]  {$\kl{P_e}{P}$};
\draw (377,57.4) node [anchor=north west][inner sep=0.75pt]  [color={rgb, 255:red, 208; green, 2; blue, 27 }  ,opacity=1 ]  {$\kl{P}{P_m}$};
\draw (383,193.4) node [anchor=north west][inner sep=0.75pt]  [color={rgb, 255:red, 208; green, 2; blue, 27 }  ,opacity=1 ]  {$\kl{P^\star}{P_m}$};
\draw (250,205.4) node [anchor=north west][inner sep=0.75pt]  [color={rgb, 255:red, 208; green, 2; blue, 27 }  ,opacity=1 ]  {$\kl{P_e}{P^\star}$};

\end{tikzpicture}

\caption{Information projections $P_e$ and $P_m$ of $P$ onto $\mathcal{W} \rev (\mathcal{X}, \mathcal{X}^2)$ in the full support case ($\mathcal{E} = \mathcal{X}^2$) \GW{(Theorem~\ref{theorem:e-m-projections})}, Pythagorean identities (Theorem~\ref{theorem:e-m-projections}), and \GW{the} bisection property (Proposition~\ref{proposition:divergence-symmetry}).} \label{figure:information-projection}
\end{figure}
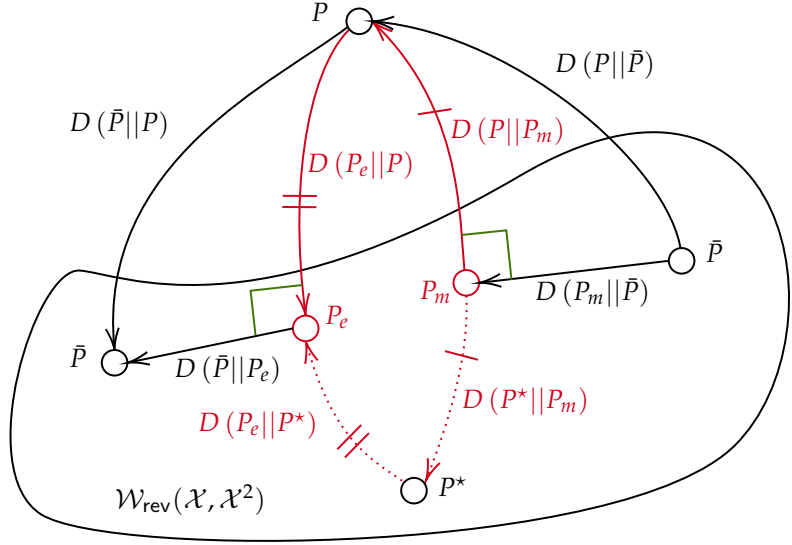

\section{The e-family of reversible edge measures}
\label{section:edge-measure-family}
Recall that $\mathcal{P}(\mathcal{X}^2)$, the set of all distributions over $\mathcal{X}^2$, forms an e-family \citep[Example~2.8]{amari2007methods}.
For some e-family of irreducible transition kernels $\mathcal{V}_e \text{ \GW{$\subset$} } \mathcal{W}(\mathcal{X}, \mathcal{X}^2)$,
one may wonder whether the corresponding family of edge measures also forms an e-family \GWS{of} distributions in $\mathcal{P}(\mathcal{X}^2)$.
We begin by illustrating that this holds in particular
for the e-family obtained by tilting a memoryless Markov kernel.

\begin{example}
\label{example:degenerate-iid}
Consider the degenerate Markov kernel corresponding to an iid process
$P(x, x') = \pi(x')$ for $\pi \in \mathcal{P}(\mathcal{X})$.
For a given function $g \colon \mathcal{X} \to \R$,
and $\theta \in \R$, construct $\widetilde{P}_\theta(x, x') = P(x, x') e^{\theta g(x')} = \pi(x') e^{\theta g(x')}$. 
Then $v_\theta = \unit$ is right eigenvector of $\widetilde{P}_\theta$ with eigenvalue $\rho(\theta) = \sum_{x' \in \mathcal{X}} \pi(x') e^{\theta g(x')}$. Letting
$\pi_\theta(x) = \pi(x) e^{\theta g(x)}/\rho(\theta)$, we see that $\pi_\theta$ is the left PF eigenvector of $\widetilde{P}_\theta$,
and the stationary distribution of the rescaled $P_\theta$. We can therefore write,
\begin{equation*}
    Q_\theta(x,x') = \exp \left(\log \pi(x)\pi(x') + \theta(g(x) + g(x')) - 2 \log \rho(\theta) \right),
\end{equation*}
thus $\set{Q_\theta}_{\theta \in \Theta}$ forms an exponential family of distributions over $\mathcal{X}^2$.
\GW{This fact can be further understood in the following manner. 
An e-family of distributions $\set{\pi_\theta}_{\theta}$ induces an e-family of memoryless Markov kernels $\set{P_\theta}_\theta$ 
with $P_\theta(x,x') = \pi_\theta(x')$ (see Lemma~\ref{lemma:iid-e-family} for a proof of this fact for the set of all memoryless kernels), and thus with edge measures $Q_\theta(x,x') = \pi_\theta(x) \pi_\theta(x')$.
Since the 2-iid extension $\set{\pi_\theta(x)\pi_\theta(x')}_\theta$ of the e-family $\set{\pi_\theta}_\theta$ is also an e-family, it follows that
$\set{Q_\theta(x,x')}_\theta$ forms an e-family.
}
\end{example}

In the remainder of this section, we show that the subset of positive reversible edge measures $\mathcal{Q} \rev = \mathcal{Q} \rev (\mathcal{X}, \mathcal{X}^2)$, induced from the e-family of reversible positive kernels, forms a submanifold of $\mathcal{P}(\mathcal{X}^2)$ that is autoparallel with respect to the e-connection, i.e.
$\mathcal{Q}\rev$ is an e-family of distribution of over pairs.
Our proof will rely on the definition of a \emph{Markov map}.

\begin{definition}[e.g. \citet{nagaoka2017information}]
\label{definition:markov-map}
We say that $M \colon \mathcal{P}(\mathcal{X}) \to  \mathcal{P}(\mathcal{Y})$ is a Markov map, when there exists a transition kernel $P_M$ from $\mathcal{X}$ to $\mathcal{Y}$ (also called a channel) such that for any $\mu \in \mathcal{P}(\mathcal{X})$, 
$$M(\mu) = \sum_{x \in \mathcal{X}} P_M(x, \cdot) \mu(x).$$
Let $\mathcal{U}$ and $\mathcal{V}$ be smooth submanifolds (statistical models) of $\mathcal{P}(\mathcal{X})$ and $  \mathcal{P}(\mathcal{Y})$ respectively.
When there exists a pair of Markov maps $M \colon \mathcal{P}(\mathcal{X}) \to  \mathcal{P}(\mathcal{Y})$, 
$N \colon \mathcal{P}(\mathcal{Y}) \to  \mathcal{P}(\mathcal{X})$ 
such that their restrictions $M|_\mathcal{U}$, $N|_\mathcal{V}$ are bijections between $\mathcal{U}$ and $\mathcal{V}$,
and are the inverse mappings of each other, we say that $\mathcal{U}$ and $\mathcal{V}$ are Markov equivalent, and write $\mathcal{U} \cong \mathcal{V}$.
\end{definition}

\begin{lemma}
\label{lemma:qrev-markov-map}
It holds that
$$\mathcal{Q} \rev (\mathcal{X}, \mathcal{X}^2) \cong \mathcal{P}\left( \left[ \frac{ \abs{\mathcal{X}}\left(  \abs{\mathcal{X}} + 1 \right)}{2} \right] \right).$$
\end{lemma}

\begin{proof}
Identify $\mathcal{X} = [m]$, and consider $Q \in \mathcal{Q}\rev$ such that
\begin{equation*}
    \begin{split}
    Q = \begin{pmatrix}
        \eta_{11} & \frac{\eta_{12}}{2} & \frac{\eta_{13}}{2} & \dots &  \frac{\eta_{1m}}{2} \\
        \frac{\eta_{12}}{2} & \eta_{22} & \frac{\eta_{23}}{2} & \dots &  \frac{\eta_{2m}}{2} 
        \\
        \vdots & & \ddots &  &  \vdots \\
        \vdots & & & \eta_{(m-1)(m-1)} &  \frac{\eta_{(m-1)m}}{2} \\
        \frac{\eta_{1m}}{2} & \dots & \frac{\eta_{23}}{2} & \dots & \eta_{mm}, \\
    \end{pmatrix}
    \end{split}
\end{equation*}
and where $\eta_{mm} = 1 - \sum_{i \leq j, (i,j) \neq (m, m)} \eta_{ij}$.
We flatten the definition of $Q$.
$$Q = \left(\eta_{11}, \eta_{22}, \dots, \eta_{mm}, \frac{\eta_{12}}{2}, \frac{\eta_{12}}{2}, \dots, \underbrace{\frac{\eta_{ij}}{2}, \frac{\eta_{ij}}{2},}_{(i,j) \colon i < j} \dots, \frac{\eta_{(m-1)m}}{2}, \frac{\eta_{(m-1)m}}{2} \right).$$
Let the matrix $E$ with $m(m-1)/2$ columns and $m(m-1)$ rows be such that,
\begin{equation*}
    \begin{split}
    E \trn =\begin{pmatrix}
        1 &  1   & 0  & 0  & \dots & 0 & 0 & 0 & 0 \\
        0 &  0   & 1  & 1  & \dots & 0 & 0 & 0 & 0 \\
          0  & 0 &0    & 0 & \dots & 1 & 1 & 0 & 0 \\
         0   & 0 &     0 & 0 & \dots & 0 & 0 & 1 & 1 \\
    \end{pmatrix}  .
    \end{split}
\end{equation*}
Block matrix multiplication yields
$$Q \begin{psmallmatrix} I_m & 0 \\ 0 & E \end{psmallmatrix} = \left(\eta_{11}, \eta_{22}, \dots, \eta_{mm}, \eta_{12}, \dots, \eta_{ij}, \dots, \eta_{(m-1)m} \right) \in \mathcal{P}\left( \left[ \frac{m(m+1)}{2} \right] \right),$$
and further observing that for $F = \frac{1}{2} E \trn$,
it holds that $FE = I_{m(m-1)/2}$. Thus
the mappings defined by $\begin{psmallmatrix} I_m & 0 \\ 0 & E \end{psmallmatrix}$ and $\begin{psmallmatrix} I_m & 0 \\ 0 & \frac{1}{2}E \trn \end{psmallmatrix}$ are Markov maps and verify
$\begin{psmallmatrix} I_m & 0 \\ 0 & \frac{1}{2}E \trn \end{psmallmatrix} \begin{psmallmatrix} I_m & 0 \\ 0 & E \end{psmallmatrix} = I_{m(m+1)/2}$.
This finishes proving the claim.
\end{proof}

\begin{theorem}
\label{theorem:edge-measures-e-family}
The set $\mathcal{Q}\rev $ forms an e-family and an m-family of $\mathcal{P}(\mathcal{X}^2)$ with dimension $\abs{\mathcal{X}}(\abs{\mathcal{X}} + 1)/2 - 1$.
Moreover, 
$\mathcal{Q}$ does not form an e-family in $\mathcal{P}(\mathcal{X}^2)$ (except when $\abs{\mathcal{X}} = 2$).
\end{theorem}

\begin{proof}
Since $\mathcal{Q}\rev \text{ \GW{$\subset$} } \mathcal{P}(\mathcal{X}^2)$, the claim stems from  the equivalence between $(i)$ and $(ii)$ of \citet[Theorem~1]{nagaoka2017information},
and application of Lemma~\ref{lemma:qrev-markov-map}, and the fact that
$\dim \mathcal{P} \left( \left[ \frac{\abs{\mathcal{X}}(\abs{\mathcal{X}} + 1)}{2} \right]\right) = \frac{\abs{\mathcal{X}}(\abs{\mathcal{X}} + 1)}{2} - 1$.
In order to prove that $\mathcal{Q}$ is not an e-family in $\mathcal{P}(\mathcal{X}^2)$, 
we first construct the following family of edge measures over three states.
\begin{equation*}
\begin{split}
Q_0^{(3)} = \frac{1}{13}\begin{pmatrix} 1 & 1 & 2 \\ 2 & 1 & 2 \\ 1 & 3 & 1 \end{pmatrix}, \qquad
Q_1^{(3)} = \frac{1}{13}\begin{pmatrix} 1 & 3 & 1 \\ 2 & 1 & 2 \\ 2 & 1 & 1 \end{pmatrix}.
\end{split}
\end{equation*}
Computing the point on the e-geodesic in $\mathcal{P}(\mathcal{X}^2)$ at parameter value $1/2$, yields
\begin{equation*}
\begin{split}
Q_{1/2}^{(3)} \propto \begin{pmatrix} 1 & \sqrt{3} & \sqrt{2} \\ 2 & 1 & 2 \\ \sqrt{2} & \sqrt{3} & 1 \end{pmatrix},
\end{split}
\end{equation*}
which does not belong to $\mathcal{Q}$.
We can readily expand the above example to general state space size, $m > 3$, by considering the one-padded versions of the above
$Q_i^{(m)} \propto \begin{psmallmatrix} Q_i^{(3)} & \unit_3 \trn \unit_{m - 3} \\ \unit_{m - 3} \trn \unit_{3} & \unit_{m - 3} \trn \unit_{m - 3} \end{psmallmatrix}$,
for $i \in \set{0, 1}$.
\end{proof}

\begin{remark}
~\begin{enumerate}
    \item[$(i)$] \citet[Theorem~1-$(iv)$]{nagaoka2017information},
actually proves the stronger result that
$\mathcal{Q} \rev$ forms an $\alpha$-family in $\mathcal{P}(\mathcal{X}^2)$, for any $\alpha \in \R$ (see \citet[Section~2.6]{amari2007methods} for a definition of $\alpha$-families).
    \item[$(ii)$] We note but do not pursue here the fact that a more refined treatment over some irreducible edge set $\mathcal{E} \subsetneq \mathcal{X}^2$ is possible.
\end{enumerate}
\end{remark}

\section{Comparison of remarkable families of Markov chains}
\label{section:comparison-usual-classes}

We briefly compare the geometric properties of reversible kernels with that of several other remarkable families of Markov chains, and compile a summary in Table~\ref{table:submanifold-summary}.

\paragraph{Family of all kernels irreducible over $(\mathcal{X}, \mathcal{E})$: $\mathcal{W}(\mathcal{X}, \mathcal{E})$.}

This family is known to form both an e-family and an m-family of dimension $\abs{\mathcal{E}} - \abs{\mathcal{X}}$ \citep[Corollary~1]{nagaoka2017exponential}.

\paragraph{Family of all reversible kernels irreducible over $(\mathcal{X}, \mathcal{E})$: $\mathcal{W} \rev (\mathcal{X}, \mathcal{E})$.}

We show in Theorem~\ref{theorem:quotient-space-dimension-e-family} and Theorem~\ref{theorem:doubly-autoparallel-submanifold} that $\mathcal{W} \rev (\mathcal{X}, \mathcal{E})$ is both an e-family and m-family or dimension $T(\mathcal{E})$,
where
\GW{$$\abs{T(\mathcal{E})} = \frac{\abs{\mathcal{E}}+ \abs{T_0(\mathcal{E})}}{2}  - 1,$$
with $T_0(\mathcal{E}) = \set{ (x,x') \in \mathcal{E} \colon x = x' }$.}

\paragraph{Family of positive memoryless (iid) kernels: $\mathcal{W} \iid (\mathcal{X}, \mathcal{X}^2)$.}

This family comprises degenerate irreducible kernels that correspond to iid processes, i.e. where all rows are equal to the stationary distribution.
Notice that for $P \in \mathcal{W} \iid$, irreducibility forces $P$ to be positive. We show that $\mathcal{W} \iid$ is an e-family of dimension $\abs{\mathcal{X}} - 1$ (Lemma~\ref{lemma:iid-e-family}), but not an m-family (Lemma~\ref{lemma:iid-not-m-family}).

\begin{lemma}
\label{lemma:iid-e-family}
$\mathcal{W} \iid$ forms an e-family of dimension $\abs{\mathcal{X}} - 1$.
\end{lemma}

\begin{proof}
For $\mathcal{X} = [m]$, let us consider the following parametrization proposed by \citet{ito1988}:
\begin{equation*}
\begin{split}
    \log P(x,x') = &\sum_{i = 1}^{m - 1} \log \frac{P(m, i)P(i, m)}{P(m, m)P(m, m)} \delta_i(x') \\
    &+ \sum_{i = 1}^{m-1} \sum_{j = 1}^{m-1} \log \frac{P(i, j) P(m,m)}{P(m,j)P(i,m)} \delta_i(x) \delta_j(x') \\
    &+ \log P(x, m) - \log P(x', m) + \log P(m, m).
\end{split}
\end{equation*}
This corresponds to the basis
\begin{equation*}
\begin{split}
g_i &= \unit \trn \delta_i, \qquad i \in [m-1], \\
g_{ij} &= \delta_i \trn \delta_j, \qquad i,j \in [m-1], \\
\end{split}
\end{equation*}
with parameters
\begin{equation*}
\begin{split}
\theta^i &= \log \frac{P(m, i)P(i, m)}{P(m,m)P(m,m)}, \qquad \theta^{ij} = \log \frac{P(i, j)P(m, m)}{P(m,j)P(i,m)}. \\
\end{split}
\end{equation*}
Let $P$ irreducible with stationary distribution $\pi$.
Suppose first that $P$ is memoryless, i.e. for all $x, x' \in \mathcal{X}$, $P(x,x') = \pi(x')$.
In this case, for all $i,j \in [m - 1]$, the coefficient $\theta^{ij}$ vanishes, and for all $i \in [m - 1]$, it holds that $\theta^i = \pi(i)/\pi(m)$, so that we can write more simply
\begin{equation*}
\begin{split}
    \log P(x,x') = &\sum_{i = 1}^{m - 1} \log \frac{\pi(i)}{\pi(m)} \delta_i(x')  + \log \pi(m).
\end{split}
\end{equation*}
Conversely, now suppose that $\theta^{i j} = 0$ for any $i,j \in [m - 1]$.
Then the matrix
$$ \widetilde{P}(x,x') = \exp \left( \sum_{i = 1}^{m - 1} \log \frac{P(m, i)P(i, m)}{P(m, m)P(m, m)} \delta_i(x') \right)$$
has rank one, the right PF eigenvector is constant, and $P$ is memoryless.
As a result, $\mathcal{W} \iid$ is an e-family of $\mathcal{W}$ such that $\theta^{ij} = 0$ for every $i,j \in [m - 1]$.
\end{proof}

\begin{lemma}
\label{lemma:iid-not-m-family}
$\mathcal{W} \iid$ does not form an m-family.
\end{lemma}

\begin{proof}
We prove the case $\abs{\mathcal{X}} = 2$ and $p \neq 1/2$,
\begin{equation*}
\begin{split}
P_0 = \begin{pmatrix} p & 1 - p \\ p & 1 - p \end{pmatrix}, \qquad P_1 = \begin{pmatrix}  1 - p & p \\ 1 - p & p \end{pmatrix}.
\end{split}
\end{equation*}
Computing the corresponding edge measures,
\begin{equation*}
\begin{split}
Q_0 = \begin{pmatrix} p^2 & p(1 - p) \\ p(1 - p) & (1 - p)^2 \end{pmatrix}, \qquad Q_1 = \begin{pmatrix}  (1 - p)^2 & p(1-p) \\ p(1 - p) & p^2 \end{pmatrix}.
\end{split}
\end{equation*}
But then if we let
\begin{equation*}
\begin{split}
Q_{1/2} = \frac{1}{2} Q_{0} + \frac{1}{2} Q_1 = \begin{pmatrix} \frac{1}{2}(p^2 + (1- p)^2) & p(1 - p) \\ p(1 - p) & \frac{1}{2}(p^2 + (1- p)^2) \end{pmatrix},
\end{split}
\end{equation*}
we see that the stationary distribution is $\pi_{1/2} = \unit / 2$, and 
$$P_{1/2} = \begin{pmatrix} p^2 + (1 - p)^2 & 2p(1-p) \\ 2p(1-p) & p^2 + (1 - p)^2 \end{pmatrix}.$$ 
But for $p \neq 0$, $P_{1/2}$ does not belong to $\mathcal{W} \iid $, hence 
the family is not an m-family.
The proof can be extended to the more general $\mathcal{X} = [m], m >2 $ by considering instead the two kernels defined by $\pi_{p} = ( p, 1 - p, 1, \dots, 1 )/(m - 1)$ and $\pi_{1 - p}$ for $p \in (0, 1), p \neq 1/2$.
\end{proof}

For simplicity, in the remainder of this section, we mostly consider the full support case.

\paragraph{Family of positive doubly-stochastic kernel: $\mathcal{W} \bis  (\mathcal{X}, \mathcal{X}^2)$.}

Recall that a kernel $P$ is said to be doubly-stochastic, or bi-stochastic, when $P$ and $P \trn$ are both stochastic matrices. In this case, the stationary distribution is always uniform.
It is known that the set of doubly stochastic Markov chains forms an m-family of dimension $(\abs{\mathcal{X}} - 1)^2$ \citep[Example~4]{hayashi2014information}.
However, as a consequence of Lemma~\ref{lemma:symmetric-not-exponential}, it does not form an e-family (except when $\abs{\mathcal{X}} = 2$).

\paragraph{Family of positive symmetric kernel: $\mathcal{W} \sym (\mathcal{X}, \mathcal{X}^2)$.}

A Markov kernel is symmetric, when $P = P \trn$,
hence this family lies at the intersection between reversible and doubly-stochastic families of Markov kernels, which are both m-families. This implies that symmetric kernels also form
an m-family. In fact, Lemma~\ref{lemma:symmetric-kernels-is-m-family} shows that the dimension of this family is $\abs{\mathcal{X}}(\abs{\mathcal{X}} - 1)/2$.  Lemma~\ref{lemma:symmetric-not-exponential}, however, shows that $\mathcal{W} \sym$ only forms an e-family for $\abs{\mathcal{X}} = 2$.

\begin{lemma}
\label{lemma:symmetric-kernels-is-m-family}
$\mathcal{W} \sym$ forms an m-family of dimension $\abs{\mathcal{X}}(\abs{\mathcal{X}} - 1)/2$.
\end{lemma}

\begin{proof}
To prove the claim, we will rely on Definition~\ref{definition:m-family-parametric}-$(ii)$ of a mixture family.
Consider the functions $s_0 \colon \mathcal{X}^2 \to \R$ and 
$s_{ij} \colon \mathcal{X}^2 \to \R$ for $i,j \in \mathcal{X}, i > j$ such that for any $x,x' \in \mathcal{X}$, $s_0(x,x') = \delta_x(x')/\abs{\mathcal{X}}$ and 
$s_{ij} = \delta_i \trn \delta_j + \delta_j \trn \delta_i - 2 \delta_i \trn \delta_i$.
Let $Q \in \mathcal{Q}$, we verify that for any $x,x' \in \mathcal{X}$,
\begin{equation*}
    \begin{split}
    Q(x,x') &= s_0(x,x') + \sum_{\substack{i,j \in \mathcal{X} \\ i > j} } s_{ij}(x,x') Q(i, j),
    \end{split}
\end{equation*}
and moreover
\begin{equation*}
    \begin{split}
    \sum_{x,x' \in \mathcal{X}} s_0(x,x') = 1, \qquad
    \sum_{x,x' \in \mathcal{X}} s_{i j}(x,x') = 0, \forall i,j \in \mathcal{X}, i > j. \\
    \end{split}
\end{equation*}
It remains to show that the $s_0, s_0 + s_{ij}$, for $i > j$, are affinely independent, or equivalently, that the $s_{ij}$, for $i > j$, are linearly independent. Let $s = \sum_{i > j} \alpha_{ij} s_{ij}$ with $\alpha_{ij} \in \R$, for any $i > j$, be such that $s = 0$. For any $i > j$, taking $x = i, x'=j$ yields $\alpha_{ij} = 0$, thus the family is independent, hence constitutes a basis, and the dimension is $\abs{\set{i,j \in \mathcal{X} \colon i > j}} = \abs{\mathcal{X}}(\abs{\mathcal{X}} - 1) / 2$.
\end{proof}

\begin{lemma}
\label{lemma:symmetric-not-exponential}
For \GWS{$\abs{\mathcal{X}} \geq 2$}, 
\begin{enumerate}
    \item[$(i)$] 
The set $\mathcal{W} \sym$ does not form an e-family, \GWS{unless $\abs{\mathcal{X}} = 2$.}
    \item[$(ii)$]
The set $\mathcal{W} \bis$ does not form an e-family, \GWS{unless $\abs{\mathcal{X}} = 2$.}
\end{enumerate}

\end{lemma}

\begin{proof}
\GWS{We first treat the case $\abs{\mathcal{X}} = 2$ for $(i)$ and $(ii)$.
Notice that 
$$P_\theta = \begin{pmatrix} \frac{e^\theta}{1 + e^\theta} & \frac{1}{1 + e^\theta} \\ \frac{1}{1 + e^\theta} & \frac{e^\theta}{1 + e^\theta} \end{pmatrix}$$ for $\theta \in \R$ satisfies $P_\theta \in \mathcal{W}\sym$, and that the latter expression exhausts all irreducible symmetric chains. We can therefore write 
    $$\mathcal{W} \sym = \set{P_\theta \colon P_\theta(x,x') = \exp\left( \delta_x(x')  \theta - \log (e^\theta + 1) \right), \theta \in \R },$$
    which follows the defintion at \eqref{eq:e-family-expression} of an e-family with carrier kernel $K = 0$, generator $g(x,x') = \delta_x(x')$, natural parameter $\theta$, $R_\theta = 0$ and potential function $\psi_\theta = \log(e^\theta + 1)$.
Furthermore, for $\abs{\mathcal{X}} = 2$, it is easy to see that symmetric and doubly-stochastic families coincide, hence $\mathcal{W} \bis$ is also an e-family.}\\\\
We now prove $(i)$ for $\abs{\mathcal{X}} = 3$.
We will consider two positive symmetric Markov kernels $P_0$ and $P_1$, and look at the e-geodesic
$$G_e(P_0, P_1) \eqdef \set{ \stoch \left( \widetilde{P}_\theta \right) : \widetilde{P}_\theta = P_0(x, x')^{1 - \theta} P_{1}(x, x')^\theta, \theta \in [0,1] },$$
where the map $\stoch$, defined in \eqref{equation:stochastic-version}, enforces stochasticity.
The matrix $P_\theta(x, x') = \stoch(\widetilde{P}_\theta)$
is symmetric, if and only if the right eigenvector of $\widetilde{P}_\theta$ is constant.
This, in turn, is equivalent to the rows of $\widetilde{P}_\theta$ being all equal.
Consider the two symmetric kernels
$$P_0 = \begin{pmatrix} \alpha & 2/3 - \alpha & 1/3 \\ 2/3 - \alpha & \alpha & 1/3 \\ 1/3 & 1/3 & 1/3 \end{pmatrix}, 
\; \; P_1 = \frac{1}{3}\begin{pmatrix} 1 & 1 & 1 \\
1 & 1 & 1 \\
1 & 1 & 1 \\
\end{pmatrix},$$
with free parameter $\alpha \neq 1/3$, 
and let us inspect the curve at parameter $\theta = 
1/2$. 
For $P_{1/2}$ to be symmetric, it is necessary that
$$\sqrt{\alpha} + \sqrt{2/3 - \alpha} = 2\sqrt{1/3},$$
whose unique solution is precisely $\alpha = 1/3$. 
Invoking \citet[Corollary~3]{nagaoka2017exponential}
finishes proving $(i)$ for $\abs{\mathcal{X}} = 3$.
We extend the proof to $\abs{\mathcal{X}} \geq 4$ using the padding argument of Theorem~\ref{theorem:edge-measures-e-family}, considering $\frac{1}{m}\begin{psmallmatrix} 3 P_i & \unit \trn \unit \\ \unit \trn \unit & \unit \trn \unit \end{psmallmatrix}$.
Suppose for contradiction that $(ii)$ is false, i.e. bi-stochastic matrices form an e-family. Take then any e-geodesic between two arbitrary symmetric kernels. The latter operators being reversible, so is the geodesic. But then this curve must also be composed entirely of symmetric matrices, hence the geodesic is symmetric, which contradicts $(i)$.
\end{proof}

\begin{remark}
For $\abs{\mathcal{X}} \geq 3$, the following hierarchies hold:

$$\mathcal{W} \iid \stackrel{\text{e-family}}{\subsetneq} \mathcal{W} \rev \stackrel{\text{e-family}}{\subsetneq} \mathcal{W},$$
$$\mathcal{W} \sym \stackrel{\text{m-family}}{\subsetneq} \mathcal{W} \rev  \stackrel{\text{m-family}}{\subsetneq} \mathcal{W},$$
$$ \mathcal{W} \sym  \stackrel{\text{m-family}}{\subsetneq} \mathcal{W} \bis  \stackrel{\text{m-family}}{\subsetneq} \mathcal{W}.$$
\end{remark}

\begin{table}[H]
\caption{Summary of geometric properties of submanifolds of irreducible Markov kernels ($\abs{\mathcal{X}} \geq 3$). We also include, for completeness, the one dimensional manifolds defined by the e-geodesic $G_e (P_0, P_1)$ and  m-geodesic $G_m (P_0, P_1)$ between two irreducible kernels $P_0$ and $P_1$, as defined in Section~\ref{section:reversible-geodesics}. Note that for the binary case $\abs{\mathcal{X}} = 2, \mathcal{W} \sym= \mathcal{W} \bis$ forms an e-family.}
\begin{center}
    \begin{tabular}{ | c | c | c | c | c |}
    \hline
    Manifold & m-family & e-family  & Dimension \\ \hline
    $\mathcal{W}(\mathcal{X}, \mathcal{E}) $ & \cmark & \cmark &  $\abs{\mathcal{E}} - \abs{\mathcal{X}}$ %
    \\ \hline
    $\mathcal{W}(\mathcal{X}, \mathcal{X}^2) $ & \cmark & \cmark &  $\abs{\mathcal{X}}(\abs{\mathcal{X}} - 1)$ %
    \\ \hline
    $\mathcal{W} \rev (\mathcal{X}, \mathcal{E})$ & \cmark & \cmark &  $\abs{T(\mathcal{E})}$ \\ %
    \hline
    $\mathcal{W} \rev (\mathcal{X}, \mathcal{X}^2)$ & \cmark & \cmark &  $\abs{\mathcal{X}}(\abs{\mathcal{X}} + 1)/2 - 1$ \\
    \hline
    $\mathcal{W} \bis (\mathcal{X}, \mathcal{X}^2)$ & \cmark & \xmark &  $(\abs{\mathcal{X}}-1)^2$ \\
    \hline
    $\mathcal{W} \sym (\mathcal{X}, \mathcal{X}^2)$ & \cmark & \xmark &  $\abs{\mathcal{X}}(\abs{\mathcal{X}} -1 )/2 $ \\
    \hline
    $\mathcal{W} \iid (\mathcal{X}, \mathcal{X}^2)$ & \xmark & \cmark &  $\abs{\mathcal{X}} - 1$ \\
    \hline
    $G_m (P_0, P_1)$ & \cmark & \xmark  &  1 \\
    \hline
    $G_e (P_0, P_1)$ & \xmark & \cmark &   1 \\
    \hline
    \end{tabular}
\end{center}
\label{table:submanifold-summary}
\end{table}

\section{Generation of the reversible family}
\label{section:generation-by-completion}
In this final section, we consider the 
family of positive Markov kernels
$\mathcal{W} = \mathcal{W}(\mathcal{X},\mathcal{X}^2)$
i.e. where the support $\mathcal{E} = \mathcal{X}^2$. 
We first show that 
$\mathcal{W} \rev$ is in a sense the smallest exponential family that contains $\mathcal{W} \sym$, the family of symmetric Markov kernels. Our notion of minimality relies on the following definition of the exponential hull of some submanifold of $\mathcal{W}$.

\begin{definition}[Exponential hull]
\label{definition:e-hull}
Let $\mathcal{V} \text{ \GW{$\subset$} } \mathcal{W}$.
\begin{equation*}
\begin{split}
\ehull(\mathcal{V}) = \Bigg\{ \stoch(\widetilde{P}) &\colon \log [\widetilde{P}] = \sum_{i = 1}^{k} \alpha_i \log [P_i ], \\
&k \in \N, \alpha_1, \dots, \alpha_k  \in \R, \sum_{i = 1}^{k} \alpha_i = 1, P_1, \dots P_k \in \mathcal{V} \Bigg\},    
\end{split}
\end{equation*}
where $\stoch$ is defined in \eqref{equation:stochastic-version}.
\end{definition}

\GWS{Remark: When $U = \frac{1}{\abs{\mathcal{X}}} \unit \trn \unit \in \mathcal{V}$, the constraint $\sum_{i = 1}^{k} \alpha_i = 1$ is redundant. 
Indeed, since $U$ corresponds to the origin in e-coordinates, the linear hull and affine hull coincide in this case.}

\begin{theorem}
\label{theorem:e-hull-symmetric-is-reversible}
For $\abs{\mathcal{X}} \geq 3$, it holds that
$$\ehull(\mathcal{W} \sym) = \mathcal{W} \rev .$$
\end{theorem}

\begin{proof}
We begin by proving the inclusion $\ehull(\mathcal{W} \sym ) \text{ \GW{$\subset$} } \mathcal{W} \rev$.
Let $P \in \ehull(\mathcal{W} \sym )$, then
there exist a positive $\widetilde{P} \in \mathcal{F}_+$, and $k \in \N, \alpha_1, \dots, \alpha_k \in \R, P_1, \dots, P_k \in \mathcal{W} \sym$ such that
$\log \widetilde{P}(x,x') = \sum_{i = 1}^{k} \alpha_i \log P_i(x,x')$.
Observe that the function $\sum_{i = 1}^{k} \alpha_i \log P_i(x,x')$ is symmetric in $x$ and $x'$, thus $\log \widetilde{P}(x,x')$ is log-reversible, and $P$ is reversible.

We now prove the second inclusion $\mathcal{W} \rev  \text{ \GW{$\subset$} } \ehull(\mathcal{W} \sym )$.
We let\GWS{
$$\mathcal{H} = \spann \left( \set{ \log [P] \colon P \in \mathcal{W}\sym } \cup \mathcal{N} \right).$$}
Recall from Theorem~\ref{theorem:reversible-basis} that the functions $g_{ij} = \delta_i \trn \delta_j + \delta_j \trn \delta_i$ , for $(i,j) \in T(\mathcal{X}^2)$, form a basis of the quotient space $\mathcal{G} \rev  = \mathcal{F} \rev / \mathcal{N}$.
It suffices therefore to show that $\set{g_{ij} \colon (i, j) \in T(\mathcal{X}^2)} \text{ \GW{$\subset$} } \mathcal{H}$.
Introduce a free parameter $t \in (0, 1), t \neq 1/2$, and let us fix $(i,j) \in T_+(\mathcal{X})$. Consider $P_{ij,t} \in \mathcal{W} \sym $ defined as follows
\begin{equation*}
    P_{ij,t}(x,x') \eqdef \begin{cases}
    2(1 - t)/\abs{\mathcal{X}} & \text{ if } (x,x') \in \set{(i,i), (j,j)}, \\
    2t/\abs{\mathcal{X}} & \text{ if } (x,x') \in \set{(i,j), (j,i)}, \\
    1/\abs{\mathcal{X}} & \text{ otherwise, }
    \end{cases}
\end{equation*}
and the functions $\hat{h}_{ij}, \tilde{h}_{ij}$
\begin{equation*}
\begin{split}
    \hat{h}_{ij} &= \log \abs{\mathcal{X}} + \log \left[P_{ij,t} \right] = a (\delta_i \trn \delta_i + \delta_j \trn \delta_j) + b (\delta_i \trn \delta_j + \delta_j \trn \delta_i), \\
    \tilde{h}_{ij} &= \log \abs{\mathcal{X}} + \log \left[ P_{ij,1 - t} \right] = b (\delta_i \trn \delta_i + \delta_j \trn \delta_j) + a (\delta_i \trn \delta_j + \delta_j \trn \delta_i), \\
\end{split}
\end{equation*}
where for simplicity we wrote $a = \log 2(1 - t)$ and $b = \log 2t \neq a$. Since the function $((x, x') \mapsto \log \abs{\mathcal{X}}) \in \mathcal{N}$, we have 
$\hat{h}_{ij}, \tilde{h}_{ij} \in \mathcal{H}$. Notice that we can write
$$g_{ij} = \frac{b \hat{h}_{ij} - a \tilde{h}_{ij}}{b^2 - a^2},$$
hence also $g_{ij} \in \mathcal{H}$. Introduce the function
$$h_{ij} = \frac{a \hat{h}_{ij} - b \tilde{h}_{ij}}{a^2 - b^2} = \delta_i \trn \delta_i + \delta_j \trn \delta_j \in \mathcal{H},$$
and observe that we can rewrite the identity $I = \unit \trn \unit - \sum_{(i,j) \in T_+(\mathcal{X}^2)} g_{ij}$ \GW{with $\unit \trn \unit$ being a  constant function. It follows that $I \in \mathcal{H}$,} and 
for any $j \in \mathcal{X}$, we can express 
$$g_{jj} = \frac{2}{\text{\GWS{$\abs{\mathcal{X}} - 2$}}} \left( \sum_{\substack{i \in \mathcal{X} \\ i > j}} h_{ij} + \sum_{\substack{i \in \mathcal{X} \\ i < j}} h_{ji} - I \right) \in \mathcal{H}.$$
As a result, $\set{g_{ij} \colon (i,j) \in T(\mathcal{X}^2)} \text{ \GW{$\subset$} } \mathcal{H}$, and the theorem follows.
\end{proof}

\begin{remark}
Observe that in the above proof, it is crucial that $\abs{\mathcal{X}} \geq 3$. For $\abs{\mathcal{X}} =2 $, we can only have $h_{21} = \begin{psmallmatrix} 1 & 0 \\ 0 & 1 \end{psmallmatrix}$, and cannot construct $g_{11} = \begin{psmallmatrix} 1 & 0 \\ 0 & 0 \end{psmallmatrix}$ nor $g_{22} = \begin{psmallmatrix} 0 & 0 \\ 0 & 1 \end{psmallmatrix}$. This is consistent with the observation that $\ehull(\mathcal{W} \sym) \neq \mathcal{W} \rev$ for $\abs{\mathcal{X}} =2 $.
\end{remark}

Secondly, we show that 
$\mathcal{W} \rev$ is also the smallest mixture family that contains $\mathcal{W} \iid$, the family of Markov kernels that correspond to iid processes. For this, we define minimality in terms of a mixture hull.

\begin{definition}[Mixture hull]
\label{definition:m-hull}
Let $\mathcal{V} \text{ \GW{$\subset$} } \mathcal{W}$.
\begin{equation*}
\begin{split}
\mhull(\mathcal{V}) = \Bigg\{ P &\colon Q \in \mathcal{Q},  Q = \sum_{i = 1}^{k} \alpha_i  Q_i, \\ &k \in \N, \alpha_1, \dots, \alpha_k \in \R, P_1, \dots, P_k \in \mathcal{V}\Bigg\},
\end{split}
\end{equation*}
where $Q$ (resp. $Q_i$) pertains to $P$ (resp. $P_i$).
\end{definition}

\begin{theorem}
\label{theorem:m-hull-iid-is-reversible}
It holds that
$$\mhull(\mathcal{W} \iid ) = \mathcal{W} \rev.$$
\end{theorem}

\begin{proof}

Let $P \in \mhull(\mathcal{W} \iid)$, then the corresponding edge measure can be expressed as a linear combination 
$\sum_{i = 1}^{k} \alpha_i Q_i$, with $k \in \N, \alpha_1, \dots, \alpha_k \in \R$, and where the $Q_i$ pertain to some degenerate iid kernel $P_i = \unit \trn \pi_i$. This implies that $Q_i(x,x') = \pi_i(x) \pi_i(x')$, hence $Q_i$ is symmetric. In turn, $Q$ is symmetric, i.e. $P$ is reversible, and $\mhull(\mathcal{W} \iid ) \text{ \GW{$\subset$} } \mathcal{W} \rev $.\\\\
For $(i,j) \in \mathcal{X}^2$, $i \geq j$, and $\eps \in [0, 1]$, consider the mixture distribution 
$$\pi_{ij, \eps} = \frac{\eps}{\abs{\mathcal{X}}} \unit + (1 - \eps) \frac{\delta_i + \delta_j}{2} \in \mathcal{P}(\mathcal{X}).$$
A direct computation yields that the pair probabilities of the iid process can be written as
\begin{equation*}
\begin{split}
Q_{i j, \eps}(x, x') = & \frac{\eps^2}{\abs{\mathcal{X}}^2} + \frac{\eps(1 - \eps)}{2 \abs{\mathcal{X}}} \left\{\delta_i(x) + \delta_j(x) + \delta_i(x') + \delta_j(x') \right\} \\ &+ \frac{(1 - \eps)^2}{4} \left\{ \delta_i(x)\delta_i(x') + \delta_i(x)\delta_j(x') + \delta_j(x)\delta_i(x') + \delta_j(x)\delta_j(x')\right\}.
\end{split}
\end{equation*}
\GWS{We first show that $\set{Q_{ij,0} \colon i \geq j}$ forms a basis of $\mathcal{F}\sym$.
    Let $\set{\alpha_{ij} \in \R \colon i \geq j}$ be such that
    $\sum_{i \geq j} \alpha_{ij} Q_{ij, 0} = 0$.
    Consider first $x, x' \in \mathcal{X}$ such that $x > x'$.
    $$\sum_{i \geq j} \alpha_{ij} Q_{ij, 0}(x,x') = \frac{1}{4}\sum_{i \geq j} \alpha_{ij} \delta_i(x)\delta_j(x') = \frac{1}{4}\alpha_{xx'} = 0 \text{ and } \alpha_{xx'} = 0.$$
    By a similar argument for the case $x < x'$, we obtain that $\alpha_{xx'} = 0$ for any $x \neq x'$.
    Inspecting now the diagonal for $x \in \mathcal{X}$, 
    $$\sum_{i \geq j} \alpha_{ij} Q_{ij, 0}(x,x) = \sum_{i \in \mathcal{X}} \alpha_{ii} Q_{ii, 0}(x,x) = \sum_{i \in \mathcal{X}} \alpha_{ii} \delta_i(x) = \alpha_{xx} = 0.$$
    This implies that the family $\set{Q_{ij, 0} \colon i \geq j}$ is independent. Since $\dim \mathcal{F} \sym = \abs{\mathcal{X}}(\abs{\mathcal{X}} + 1) /2 = \abs{\set{Q_{ij, 0} \colon i \geq j }}$, it is maximally so, thus forms a basis.
    However, the basis elements are not in $\mathcal{W} \iid$. We therefore examine the case $\eps > 0$, and leverage the property that in normed vector spaces, finite linearly independent systems are stable under small perturbations (see Lemma~\ref{lemma:stability-basis} reported below for convenience) in order to show existence of a basis in $\mathcal{W} \iid$.
    \begin{lemma}[{\citet[p.9, Exercise~35]{costara2013exercises}}]
    \label{lemma:stability-basis}
    Let $n \in \N$, $X$ is a normed vector space and $x_1, \dots, x_n$ are $n$ linearly independent elements in $X$. Then there exists $\eta > 0$ such that if $y_1, y_2, \dots, y_n$ are such that $\nrm{y_i} < \eta$ for $i = 1, \dots, n$, then $x_1 + y_1, x_2 + y_2, \dots, x_n + y_n$ are also $n$ linearly independent elements in $X$.
    \end{lemma}
    Let us consider $(\mathcal{F} \sym, \nrm{\cdot}_{1,1})$, the space of real symmetric matrices equipped with the \emph{entry-wise} $\ell_1$ norm. For any $i \geq j$ and for any $\eps \in (0, 1)$,
    \begin{equation*}
    \begin{split}
        & \nrm{Q_{ij, \eps} - Q_{ij, 0}}_{1,1} \eqdef
\sum_{x,x' \in \mathcal{X}} \abs{ Q_{ij, \eps}(x,x') - Q_{ij, 0}(x,x') } \\  
&\leq \abs{\mathcal{X}}^2 \abs{\frac{\eps^2}{\abs{\mathcal{X}}^2}} + \abs{\frac{\eps(1 - \eps)}{2 \abs{\mathcal{X}}}} \sum_{x,x'} \abs{ \delta_i(x) + \delta_j(x) + \delta_i(x') + \delta_j(x')} \\
& + \abs{\frac{\eps(2 - \eps)}{4}} \sum_{x,x' \in \mathcal{X}} \abs{ \delta_i(x)\delta_i(x') + \delta_i(x)\delta_j(x') + \delta_j(x)\delta_i(x') + \delta_j(x)\delta_j(x')} \\
&\leq \eps^2 + 2 \eps \abs{\mathcal{X}} + 2 \eps \abs{\mathcal{X}}^2,
\end{split}
    \end{equation*}
thus $$\nrm{Q_{ij, \eps} - Q_{ij, 0}}_{1,1} \leq 5 \eps \abs{\mathcal{X}}^2.$$
Let $\eta$ as defined in Lemma~\ref{lemma:stability-basis}, with respect to the basis $\set{Q_{ij,0} \colon i \geq j}$,
and choose $0 < \eps < \frac{\eta}{5 \abs{\mathcal{X}}^2}$. 
Then $\nrm{Q_{ij, \eps} - Q_{ij, 0}}_{1,1} < \eta$, thus the family $\set{ Q_{ij, \eps \colon i \geq j} }$ is a also basis for $\mathcal{F} \sym$ that lies in $\mathcal{W} \iid$, whence the theorem.}
\end{proof}

\section*{Acknowledgements}
We thank the anonymous referees for the helpful comments, which helped us to improve the presentation of this manuscript.
We also thank Hiroshi Nagaoka for an enlightening discussion, and constructive remarks.

\bibliography{bibliography}
\bibliographystyle{abbrvnat}

\end{document}